\documentclass{amsart}

\usepackage{mathrsfs}
\usepackage{url}
\usepackage{verbatim}
\usepackage[left=1.2in, right=1.2in, top=1.2in, bottom=1.2in]{geometry}
\usepackage{amsfonts}
\usepackage{cleveref}
\usepackage{amssymb}
\usepackage{amsmath}
\usepackage{amsthm}
\newcommand{\triplearrows}{\begin{smallmatrix} \to \\ \to \\ 
\to \end{smallmatrix} }

\newcommand{\alg}{\mathrm{Alg}}
\newcommand{\ko}{{ko}}

\renewcommand{\hom}{\mathrm{Hom}}
\newcommand{\sOGF}{\mathcal{O}_{\mathcal{F}}(G)}
\newcommand{\ho}{\mathrm{Ho}}
\newcommand{\e}[1]{\mathbb{E}_{#1}}
\newcommand{\cati}{\mathrm{Cat}_\infty}
\newcommand{\cb}{\mathrm{CB}^\bullet}
\newcommand{\md}{\mathrm{Mod}}
\newcommand{\clg}{\mathrm{CAlg}}
\newcommand{\fun}{\mathrm{Fun}}
\newcommand{\catst}{\mathrm{Cat}_\infty^{\mathrm{perf}}}

\usepackage{xy}
\input xy
\newcommand{\spec}{\mathrm{Spec}}
\renewcommand{\sp}{\mathrm{Sp}}
\xyoption{all}

\renewcommand{\mod}{\mathrm{Mod}}
\newcommand{\sF}{\mathcal{F}}
\theoremstyle{definition}
\newtheorem{definition}{Definition}[section]

\newtheorem{question}[definition]{Question}
\newtheorem{example}[definition]{Example}
\newtheorem{proposition}[definition]{Proposition}
\newtheorem{cons}[definition]{Construction}

\newtheorem{corollary}[definition]{Corollary}

\newtheorem{remark}[definition]{Remark}
\newtheorem{theorem}[definition]{Theorem}

\newcommand{\pic}{\mathrm{Pic}}

\begin{document}
	\title{Examples of descent up to nilpotence}
	\author{Akhil Mathew}
	\address{Harvard University\\ Cambridge, Massachusetts, USA }
	\email{amathew@math.harvard.edu}
	\urladdr{http://math.harvard.edu/~amathew/}
	\date{\today}

\maketitle

\begin{abstract}
We give a survey of the ideas of descent and nilpotence.
We focus on examples arising from chromatic homotopy theory and from group
actions, as well as a few examples in algebra.  
\end{abstract}

\section{Introduction}

Let $(\mathcal{C}, \otimes, \mathbf{1})$ denote a \emph{tensor-triangulated}
category, i.e., a triangulated category equipped 
with a compatible (in particular, biexact) symmetric monoidal structure. 
Examples of such abound in various aspects both of stable homotopy theory
(e.g., the stable homotopy category) and
in representation theory, via derived 
categories of representations and stable module categories. 
In many such cases, we are interested in describing large-scale features of
$\mathcal{C}$ and of the associated mathematical structure. 
To this end, there are a number of basic invariants of $\mathcal{C}$ that we
can study, such as the lattice of thick subcategories (or, preferably, thick
$\otimes$-ideals)
or localizing subcategories, 
the Picard group $\pic(\mathcal{C})$ of isomorphism classes of
invertible objects, and 
the Grothendieck group $K_0(\mathcal{C})$. 

One of our basic goals 
is relate invariants of $\mathcal{C}$ to those of another
$\otimes$-triangulated category $\mathcal{C}'$ receiving a
$\otimes$-triangulated functor $\mathcal{C} \to \mathcal{C}'$; for whatever
reason, we may expect invariants to be simpler to understand for $\mathcal{C}'$.
However, we can also hope to use information over $\mathcal{C}'$ to understand
information over $\mathcal{C}$ via ``descent.''
Consider for instance the extension $\mathbb{R} \subset \mathbb{C}$ of the real
numbers to the complex numbers. 
Since $\mathbb{C}$ is algebraically closed, phenomena are often much easier to
study over $\mathbb{C}$ than $\mathbb{R}$. However, we can study phenomena over
$\mathbb{C}$ via the classical process of Galois descent. 

Classically, in algebra, descent is carried out along \emph{faithfully flat}
maps of rings. Here, however, it turns out that there is a large class of
extensions which are far from faithfully flat, but which satisfy a categorical
condition that forces the conclusion of descent nonetheless to hold.  
The key definition is as follows.

\begin{definition} 
Let $\mathcal{C}$ be a $\otimes$-triangulated category and let $A$ be an
algebra object. An object of $\mathcal{C}$ is said to be \emph{$A$-nilpotent}
if it belongs to the thick $\otimes$-ideal generated by $A$.  
The algebra $A$ is said to be \emph{descendable} if the unit of $\mathcal{C}$
is $A$-nilpotent. We will say that a map $R \to R'$ of $\e{\infty}$-ring
spectra is \emph{descendable} if $R'$ is descendable in
$\md(R)$, the category of $R$-module spectra. \end{definition} 

The notion of ``$A$-nilpotence'' is very classical and goes back to Bousfield
\cite{Bou79}. The idea of ``descendability'' is implicit at various points in
the literature (in particular, in the work of Hopkins-Ravenel in chromatic
stable homotopy theory) but has been systematically studied by
Balmer \cite{Balmersep} (especially in relation to the spectrum of
\cite{Balmer05}) and by the author in \cite{MGal}. 
We refer as well to \cite[Appendix D.3]{lurie_sag} for a treatment.
The notion of descendability is enough to imply that a version of faithfully
flat descent holds; however, a descendable algebra may be far from being
faithfully flat. A simple example is the map $KO \to KU$ of ring spectra.

In this paper, our goal is to give an introduction to these ideas
and an overview of several examples, emphasizing the $\sF$-nilpotence of
\cite{MNN15i, MNN15ii}. 
This paper is mostly intended as an exposition, although some of
the results are improvements or variants of older ones. 
A number of basic questions remain, and we have taken the
opportunity to highlight some of them. 
We have in particular emphasized the role that \emph{exponents} of nilpotence
(\Cref{def:nilpexp}) plays. 
One new result in this paper, which was explained to us by Srikanth Iyengar, is
\Cref{thm:flatdescsri}, which gives a larger class of faithfully flat extensions
which are descendable. 

There are many 
applications of these ideas that we shall not touch on here. For example, we
refer to \cite{MS, HMS} for the use of these techniques to calculate Picard
groups of certain ring spectra; \cite{Balmersep, Ma15} for applications
to the classification of thick subcategories; and \cite{CMNN} for
applications to Galois descent in algebraic $K$-theory.

\subsection*{Conventions}

Throughout this paper, we will use the language of $\infty$-categories as in
\cite{Lur09, Lur16}. In most
cases, we will only use the language very lightly, as most of the invariants
exist at the triangulated level.\footnote{We note that the descent-theoretic
approach to Picard groups uses $\infty$-categorical technology in an essential
manner.} 

In particular, we will use the phrase \emph{stably symmetric monoidal
$\infty$-category} to mean a (usually small) symmetric monoidal, 
stable $\infty$-category whose tensor product is biexact; this is the natural
$\infty$-categorical lift of a tensor-triangulated category. We will also need
to work with large $\infty$-categories. We will use the theory of presentable
$\infty$-categories of \cite[Ch. 5]{Lur09}. In particular, a \emph{presentably
symmetric monoidal stable $\infty$-category} is the natural ``large'' setting
for these questions; it refers to a presentable, stable $\infty$-category
equipped with a symmetric monoidal tensor product which is bicocontinuous.  

We will let $\alg(\mathcal{C})$ denote the $\infty$-category of
associative (or $\e{1}$) algebras 
in a symmetric monoidal $\infty$-category and let $\clg(\mathcal{C})$ denote
the $\infty$-category of commutative (or $\e{\infty}$) algebras. 
If $R$ is an associative (i.e., $\mathbb{E}_1$) ring spectrum, we will write
$\md(R)$ for the $\infty$-category of $R$-module spectra (see \cite{EKMM},
\cite[Ch. 7]{Lur16}). When $R$ is an
$\e{\infty}$-ring, then $\md(R)$ is a presentably symmetric monoidal stable
$\infty$-category. When $R$ is a discrete associative (resp. commutative) ring,
then we can regard $R$ (or the associated Eilenberg-MacLane spectrum) as an
$\e{1}$ (resp. $\e{\infty}$)-ring and will write either $\md(R)$ or $D(R)$ for
the $\infty$-category of $R$-module spectra, which is equivalent to the derived
$\infty$-category of $R$. 

\newcommand{\tht}{\mathrm{Thick}^{\otimes}}
\renewcommand{\th}{\mathrm{Thick}}

\subsection*{Acknowledgments} 
I would like to thank Bhargav Bhatt, Srikanth Iyengar, and Jacob Lurie for
helpful discussions. I would
especially like to thank my collaborators Niko Naumann and Justin Noel; much
of this material is drawn from  
\cite{MNN15i, MNN15ii}.  
Most of all, I would like to thank Mike Hopkins: most of these ideas
originated in his work.

\section{Thick $\otimes$-ideals and nilpotence}

\subsection{Thick subcategories and $\otimes$-ideals}
We begin by reviewing the theory of thick subcategories. As this material is
very classical, we will be brief. 

Let $\mathcal{C}$ be an idempotent-complete stable $\infty$-category. 
\begin{definition} 
A \textbf{thick subcategory} of $\mathcal{C}$ is a full subcategory 
$\mathcal{D} \subset \mathcal{C}$ satisfying the following three conditions: 
\begin{enumerate}
\item $0 \in \mathcal{D} $. 
\item If $X_1 \to X_2 \to X_3$ is a cofiber sequence (i.e., exact triangle) in
$\mathcal{C}$, and two out of three of the $\left\{X_i\right\}$ belong to
$\mathcal{D}$, then the third belongs to $\mathcal{D}$. 
\item $\mathcal{D}$ is idempotent-complete. Equivalently, if $X, Y \in \mathcal{C}$
and $X \oplus Y  \in \mathcal{D}$, we have $X, Y \in \mathcal{D}$. 
\end{enumerate}
Given a collection $\mathcal{S} \subset \mathcal{C}$ of objects, there is a
smallest thick subcategory of $\mathcal{C}$ containing $\mathcal{S}$; we will
write $\th(\mathcal{S})$ for this and call it the thick subcategory
\textbf{generated by $\mathcal{S}$}. 
\end{definition}


\begin{cons}
We suppose that $\mathcal{S}$ is closed under direct sums and suspensions
$\Sigma^i, i \in \mathbb{Z}$.
Then $\th(\mathcal{S})$ has a natural inductive
increasing filtration $\th(\mathcal{S})_0 \subset \th(\mathcal{S})_1 \subset
\dots  \subset \th(\mathcal{S})$.
This filtration is well-known in the literature. 
Compare Christensen \cite[Sec. 3.2]{Ch98}, the dimension of triangulated
categories introduced by Rouquier \cite{Rouq}, and the treatment  and the theory of
{levels} in Avramov-Buchweitz-Iyengar-Miller \cite[Sec. 2.2]{ABIM}. 
\begin{itemize}
\item $\th(\mathcal{S})_0 = \left\{0\right\}$. 
\item $\th(\mathcal{S})_1$ consists of the retracts of objects 
in $\mathcal{S}$.
\item An object $X$ belongs to $\th(\mathcal{S})_n$ if $X$ is a retract of
an object $\widetilde{X}$ such that there exists a cofiber
sequence $\widetilde{X}' \to \widetilde{X} \to \widetilde{X}''$ such that 
$\widetilde{X}' \in \th(\mathcal{S})_1$ (i.e., $\widetilde{X}'$ is a retract
of an object in $\mathcal{S}$)  and $\widetilde{X}'' \in
\th(\mathcal{S})_{n-1}$.
\end{itemize}
\end{cons}

\begin{proposition} 
\begin{enumerate}
\item  
$\th(\mathcal{S}) = \bigcup_{n \geq 0} \th(\mathcal{S})_n$ and each $\th(\mathcal{S})_n$ is idempotent-complete.
\item
Given a cofiber sequence $X' \to X \to X''$ with $X' \in
\th(\mathcal{S})_{n_1}$ and $ X'' \in \th(\mathcal{S})_{n_2}$, we have $X \in
\th(\mathcal{S})_{n_1 + n_2}$.
\end{enumerate}
\end{proposition} 
\begin{proof} 
By construction, each $\th(\mathcal{S})_n$ is idempotent-complete. 
To see that $\th(\mathcal{S}) = \bigcup_{n \geq 0} \th(\mathcal{S})_n$, one
easily reduces to seeing that the union is itself a thick subcategory, which
follows from the second assertion. We are thus reduced to proving (2), which
is effectively a diagram chase. 

We prove (2) by induction on $n_1 + n_2$.
We can assume $n_1, n_2 > 0$. 
Up to adding a summand to both $X'$ and $X$, we may assume 
that there exists a cofiber sequence $Y_1 \to X' \to Y_2$ with 
$Y_1 \in \th(\mathcal{S})_{1}$ and $Y_2 \in \th(\mathcal{S})_{n_1 - 1}$. 
We consider the homotopy pushout square
\[ \xymatrix{
X' \ar[d]^{\psi_1} \ar[r]^{\phi_1} &  X \ar[d]^{\psi_2}  \\
Y_2 \ar[r]^{\phi_2} &  X \oplus_{X'} Y_2
}.\]
The cofiber of $\phi_2$ is equivalent to the cofiber of $\phi_1$, i.e., $X''$. 
It follows that we have a cofiber sequence
\[ Y_2 \to X \oplus_{X'} Y_2 \to X'',
  \]
  which shows by the inductive hypothesis that $X  \oplus_{ X'} Y_2 \in
  \th(\mathcal{S})_{n_1 + n_2 - 1}$. 
  In addition, the homotopy fibers of $\psi_1 , \psi_2$ are identified with
  $Y_1 \in \th(\mathcal{S})_1$, and the cofiber sequence
  \[ Y_1 \to X \to   X \oplus_{X'} Y_2 ,\]
  with $Y_1 \in \th(\mathcal{S})_1$,
 shows that $X \in \th(\mathcal{S})_{n_1 + n_2}$.  
\end{proof}

Let $f: X \to Y$ be a map in $\mathcal{C}$. We say that $f$ is
\emph{$\mathcal{S}$-zero} if for all $A \in \mathcal{S}$, the natural map 
\[ [A, X]_* \to [A, Y]_*  \]
is zero. 
It now follows that 
if $A' \in \th(\mathcal{S})_n$ and $f_1, \dots, f_n$ are composable
$\mathcal{S}$-zero maps in $\mathcal{C}$ with $g = f_n \circ \dots \circ
f_1: X \to Y$, then 
$g_*: [A' ,X]_* \to [A', Y]_*$ is zero. Compare \cite{Ch98} for a detailed
treatment.

General thick subcategories can be tricky to work with. 
In practice, for example for classification results, it is often convenient to restrict to thick $\otimes$-ideals. 

\begin{definition} 
Let $\mathcal{C}$ be a symmetric monoidal stable $\infty$-category. 
A \textbf{thick $\otimes$-ideal} is a thick subcategory $\mathcal{D} \subset
\mathcal{C}$ which has the following additional property: if $X \in
\mathcal{C}$ and $Y \in \mathcal{D}$, then $X \otimes Y \in
\mathcal{D}$.
Given a collection $\mathcal{S} \subset \mathcal{C}$ of objects, there is a
smallest thick $\otimes$-ideal of $\mathcal{C}$ containing $\mathcal{S}$; we will
write $\th^{\otimes}(\mathcal{S})$ for this and call it the thick
$\otimes$-ideal 
\textbf{generated by $\mathcal{S}$}. 
\end{definition} 
\begin{cons}
\label{thtfilt}
In the same fashion as $\th(\mathcal{S})$, we can define a filtration
$\{\th^{\otimes}(\mathcal{S})_n\}_{n \geq 0}$ on $\th^{\otimes}(\mathcal{S})$. 
Suppose $\mathcal{S}$ is closed under direct sums and suspensions $\Sigma^i, i
\in \mathbb{Z}$.
We can define 
\[ \th^{\otimes}(\mathcal{S})_n = \th( \mathcal{S}')_n,  \]
where $\mathcal{S}'$ consists of all objects in $\mathcal{C}$ of the form $X_1
\otimes X_2$ with $X_1, \in \mathcal{C}, X_2 \in \mathcal{S}$. One sees that
$\th^{\otimes}(\mathcal{S}) = \th(\mathcal{S}')$ and that this gives an
increasing
filtration on $\tht(\mathcal{S})$. 
\end{cons}

Given a stably symmetric monoidal, idempotent-complete $\infty$-category $\mathcal{C}$, the
classification of thick $\otimes$-ideals of $\mathcal{C}$ is an important
problem. The first major classification result was by Hopkins-Smith
\cite{HS98}, who carried this out for finite $p$-local spectra, which was
followed by work of Hopkins-Neeman, Thomason, Benson-Carlson-Rickard, and
many others. In general, the classification of thick $\otimes$-ideals can be
encapsulated in 
the spectrum of Balmer \cite{Balmer05}, a topological space built from the prime
thick $\otimes$-ideals  which determines all the thick $\otimes$-ideals (at
least under some conditions, e.g., if objects are dualizable).

We give one example of how thick $\otimes$-ideals can be constructed.
Given an object $X \in \mathcal{C}$, we write $\pi_* X = \pi_*
\hom_{\mathcal{C}}(\mathbf{1}, X)$. 
We let $R_* = \pi_* \mathbf{1} = \pi_* \hom_{\mathcal{C}}(\mathbf{1},
\mathbf{1})$. 
Then $R_*$ is a graded ring and $\pi_* X$ is a graded $R_*$-module. 

\begin{example} 
Given a homogeneous ideal $I \subset R_*$, we consider the subcategory of all $X \in \mathcal{C}$ such that
for all $Y \in \mathcal{C}$, the $R_*$-module $\pi_*(X \otimes Y)$ is $I$-power
torsion. This is a thick $\otimes$-ideal. 
\end{example}

\subsection{$A$-nilpotence}

We now come to the key definition, that of $A$-nilpotence. 
While these ideas are quite classical, for a detailed presentation of this
material along these lines, we refer in particular to \cite{MNN15i}. 

Let $\mathcal{C}$ be a stably symmetric monoidal, idempotent-complete $\infty$-category and let $A \in
\alg(\mathcal{C})$.

\newcommand{\nil}{\mathrm{Nil}}

\begin{definition}[Bousfield \cite{Bou79}] 
We say that an object $X \in \mathcal{C}$ is \textbf{$A$-nilpotent} if $X$
belongs to $\tht(A)$. We will also write $\nil_A = \tht(A)$.
Note that $\nil_A$ is also the thick subcategory of $\mathcal{C}$ generated by all
objects of the form $A \otimes Y, Y \in \mathcal{C}$, i.e., $\nil_A = \tht(A) =
\th(\left\{A \otimes Y\right\}_{Y \in \mathcal{C}})$. 
\end{definition}

We give one example of this now; more will follow later. 
\begin{example} 
Suppose $\mathcal{C} = D(\mathbb{Z})$ is the derived $\infty$-category of
abelian groups and $A =\mathbb{Z}/p\mathbb{Z}$. Then an object $X \in
\mathcal{C}$ is $A$-nilpotent if and only if there exists $n \geq 0$ such that
the map $p^n: X \to X$ is nullhomotopic. In particular, it follows that such an
object is $p$-adically complete and $p$-torsion. 
\end{example}

\begin{remark} 
Suppose $X \in \mathcal{C}$ is a \emph{dualizable} object. Then the thick
$\otimes$-ideal generated by $X$ is equal to the one generated by the algebra
object $X \otimes
X^{\vee}$. Therefore, $\tht(X) = \tht(X \otimes X^{\vee})$ is the subcategory
of $X \otimes X^{\vee}$-nilpotent objects. 
It follows that any thick $\otimes$-ideal of $\mathcal{C}$ generated by a
finite collection of dualizable objects is automatically of the form $\nil_A$
for some $A$.
\end{remark} 

Let $X \in \mathcal{C}$. 
Then we can try to approximate $X$ in $\mathcal{C}$ by objects of the form $A
\otimes X', X' \in \mathcal{C}$. 
There is always a canonical way of doing so. 

Recall that $\Delta^+$ denotes the category of 
finite linearly ordered sets, so that $\Delta^+$ is 
the union of the usual simplex category $\Delta$ with an initial object $[-1]$.
\newcommand{\cbaug}{\mathrm{CB}^{\mathrm{aug}}}
\begin{cons}
Suppose that $A \in \alg(\mathcal{C})$. 
Then we can form an augmented cosimplicial diagram $\cbaug(A)$ in $\mathcal{C}$, 
\[ \cbaug(A): \mathbf{1} \to A \rightrightarrows A \otimes A \triplearrows
\dots .  \]
The underlying cosimplicial diagram $\cb(A) \xrightarrow{\mathrm{def}}
\cbaug(A)|_{\Delta}$ is called the \emph{cobar construction} or \emph{Amitsur
complex.}
The diagram $\cbaug(A)$ admits a splitting or extra degeneracy (see
\cite[Sec. 4.7.3]{Lur16}) after tensoring with $A$. 
Compare \cite[Ch. 8]{Rog08} or \cite[Sec. 2]{MNN15i}.

\end{cons}

\begin{example} 
Let $F: \mathcal{C} \to \sp$ be a functor.
It follows that we have a Bousfield-Kan 
\cite{BK}
type 
spectral sequence
\begin{equation}  \label{GeneralANSS} E_2^{s,t} = H^s( \pi_t F( X \otimes
\cb(A))) \implies \pi_{t-s}
\mathrm{Tot}( F(X \otimes \cb(A))).  \end{equation}
We will call this  the
\emph{$A$-based Adams spectral sequence.} When one takes $\mathcal{C} = \sp$,
$F = \mathrm{id}$, and $A = H \mathbb{F}_p$, then one recovers the classical
Adams spectral sequence.  
\end{example}

The first key consequence of $A$-nilpotence is that the cobar construction
always converges very nicely. Not only does it converge (to the original
object) in $\mathcal{C}$, but it does so universally in the following sense. 
We will elaborate more on this ``universal'' convergence below. 

\begin{proposition} 
\label{Adamstowerconv}
Suppose $X \in \tht(A)$. Then for any stable $\infty$-category
$\mathcal{D}$ and exact functor $F: \mathcal{C} \to
\mathcal{D}$, 
$F(X \otimes \cbaug(A)): \Delta^+ \to \mathcal{D}$ is a limit diagram. 
That is, the natural morphism
\begin{equation}  \label{keyAmap} F(X) \to \mathrm{Tot}\left( F(X \otimes A) \rightrightarrows F(X \otimes A
\otimes A ) \triplearrows \dots \right)  
\end{equation} 
is an equivalence in $\mathcal{D}$. 
\end{proposition} 

In particular, taking $F$ to be the identity functor, it follows that the
augmented cosimplicial diagram $\Delta^+ \to \mathcal{C}$ given by 
$X \to \left( X \otimes A \rightrightarrows X \otimes A \otimes A \dots
\triplearrows \right)$ is a limit diagram. This limit diagram has the additional
property that it is preserved (i.e., remains a limit diagram) after applying
any exact functor. This is a very special property of such diagrams which we
will discuss further below. 

\begin{proof} 
The basic observation is that if $X = A \otimes Y$ for some $Y \in
\mathcal{C}$, then the augmented cosimplicial diagram $X \otimes \cbaug(A)$
has an ``extra degeneracy'' or splitting \cite[Sec. 4.7.3]{Lur16}, which in
particular implies that it is a universal limit diagram. 
It follows that 
\eqref{keyAmap} is an equivalence if $X = A \otimes Y$. Since the class of $X$
for which \eqref{keyAmap} is an equivalence is thick, it follows that this
class includes $\nil_A$. 
\end{proof}

\begin{definition} 
Let $A \in \alg(\mathcal{C})$. 
Fix $X, Y \in \mathcal{C}$.
By convention, we call the \textbf{$A$-based Adams spectral sequence} 
the $\mathrm{Tot}$ spectral sequence (or BKSS) converging to $\pi_* \hom_{\mathcal{C}}(
Y, \mathrm{Tot}( \cb(X)))$ (which coincides with $\pi_*\hom_{\mathcal{C}}(Y,
X)$ if $X \in \nil_A$). 

We will say that a map $Y \to X$ has \textbf{Adams
filtration $\geq k$} if the induced map $Y \to \mathrm{Tot}(\cb(X))$ is detected 
in filtration $\geq k$ in the BKSS. This is equivalent to the existence of the
map $Y \to X$ factoring as a $k$-fold composite of maps each of which becomes
nullhomotopic after tensoring with $A$. Compare \cite[Sec. 4]{Ch98} for a
detailed treatment (albeit with arrows in the opposite direction from our
setting).  
\end{definition}

In general, $\mathrm{Tot}$ spectral sequences  (i.e., those giving the
homotopy of a cosimplicial space or spectrum) are only ``conditionally''
convergent, with potential $\lim^1$ subtleties. 
In the nilpotent case, the spectral sequence converges in the best possible
form and is essentially finitary. 
\begin{proposition}  
\label{degeneratefinitevanishingline}
Suppose $X \in \nil_A$. 
The spectral sequence \eqref{GeneralANSS} has the following property. 
There exists $N  \geq 2$ such that at the $E_N$-page, there exists a
horizontal vanishing line of some height $h$: that is, $E_N^{s,t} =0$ for $s>
h$. 
\end{proposition} 

We will discuss this result in the next section.

\subsection{Towers}
In this subsection, we will discuss more closely the behavior of the cobar
construction of an $A$-nilpotent object. First we will need some preliminaries
about towers. 
Let $\mathcal{C}$ be an idempotent-complete stable $\infty$-category.
\newcommand{\tow}{\mathrm{Tow}}

\begin{definition} 
A tower in $\mathcal{C}$ is  a functor
$\mathbb{Z}_{\geq 0}^{op} \to\mathcal{C}$, i.e., a sequence
\[ \dots \to X_n \to X_{n-1} \to \dots \to X_1 \to  X_0 . \]
The collection of towers in $\mathcal{C}$ is naturally organized into a stable
$\infty$-category $\tow(\mathcal{C})$. 
\end{definition}

\begin{cons}
Let $X^\bullet: \Delta \to \mathcal{C}$ be a cosimplicial object. Then the
sequence of partial totalizations 
$$\mathrm{Tot}_n(X^\bullet) \xrightarrow{\mathrm{def}} \varprojlim_{[i] \in \Delta, i \leq n} X^i$$
is naturally arranged into a tower, whose inverse limit is given by the
totalization $\mathrm{Tot}(X^\bullet)$. 
A version of the Dold-Kan correspondence, due to Lurie \cite[Sec. 1.2.4]{Lur16}, implies that
the cosimplicial object can be reconstructed from the tower. In fact, one has
an equivalence of $\infty$-categories $$\fun(\Delta, \mathcal{C}) \simeq
\tow(\mathcal{C}),$$
which sends a cosimplicial object $X^\bullet$ to its $\mathrm{Tot}$ tower. 
\end{cons}

\newcommand{\town}{\mathrm{Tow}^{\mathrm{nil}}}
\newcommand{\towf}{\mathrm{Tow}^{\mathrm{fast}}}

\begin{definition} 
A tower $\left\{X_i\right\}_{i \geq 0}$ is said to be \textbf{nilpotent} if
there exists $N  \in \mathbb{Z}_{\geq 0}$ such that for each $i \in
\mathbb{Z}_{\geq 0}$, the natural map $X_{i+N} \to X_i$ arising from the tower
is nullhomotopic. We let $\town(\mathcal{C}) \subset \tow(\mathcal{C})$ denote the 
subcategory spanned by the nilpotent towers. 
\end{definition} 

We have the following straightforward result. 

\begin{proposition} 
$\town(\mathcal{C}) \subset \tow(\mathcal{C})$ is a thick subcategory.
\end{proposition}

\begin{definition} 
We will say that a tower $\left\{X_i\right\}_{i \geq 0}$ is \textbf{quickly
converging} if it belongs to the 
thick subcategory of $\tow(\mathcal{C})$ generated by the constant towers and
the nilpotent towers. The quickly 
converging towers form a thick subcategory $\towf(\mathcal{C}) \subset
\tow(\mathcal{C})$. 
We will also say that a cosimplicial object in $\mathcal{C}$ is
\textbf{quickly converging}
if the associated tower is quickly converging. 
\end{definition}

The first key observation about quick convergence is that it guarantees 
that limit diagrams are universal. 
\begin{proposition} 
\label{quickconvpreserved}
Let $X^\bullet \in \fun(\Delta, \mathcal{C})$ be a cosimplicial object.
Suppose $X^\bullet$ is quickly converging. 
Then: 
\begin{enumerate}
\item  
$\mathrm{Tot}(X^\bullet)$ exists in $\mathcal{C}$.
\item
$F(X^\bullet)$ is quickly converging in $\mathcal{D}$. 
\item 
For any idempotent-complete stable $\infty$-category $\mathcal{D}$ and exact
functor $F: \mathcal{C} \to \mathcal{D}$, 
the natural map 
\[ F( \mathrm{Tot}(X^\bullet)) \to \mathrm{Tot}( F(X^\bullet))  \]
is an equivalence.

\end{enumerate}

\end{proposition} 

That is, the limit diagram defined by $X^\bullet$ is \emph{universal} for
exact functors. In the language of pro-objects, the pro-object defined by the
tower $\{\mathrm{Tot}_n(X^\bullet)\}$ is constant. 
Compare the discussion in \cite[Sec. 3.2]{MGal}. 

\begin{proof} 
Any exact functor respects finite limits, so we can replace the totalization
by the inverse limit of the $\mathrm{Tot}$ tower. It thus suffices  
to show that if $\left\{Y_i\right\}$ is a quickly converging \emph{tower}, 
then 
the natural analogs of the three statements hold. 
That is, $\varprojlim Y_i$ exists in $\mathcal{C}$, $\left\{F(Y_i)\right\}$ is quickly converging, 
and the natural map $F( \varprojlim Y_i )\to \varprojlim F(Y_i)$ is equivalence. 
The class of towers $\left\{Z_i\right\}$ for which these three assertions is thick; it clearly contains
the constant and nilpotent towers, so it contains the quickly converging towers.
This proves the result. \end{proof}

Let $X^\bullet$ be a cosimplicial spectrum. In this case, we have the
classical Bousfield-Kan spectral sequence converging to $\pi_*
\mathrm{Tot}(X^\bullet)$. 
\begin{proposition} 
\label{quickconvss}
Let $X^\bullet$ be a quickly converging cosimplicial spectrum. 
Then the BKSS converging to $\pi_* \mathrm{Tot}(X^\bullet)$ has a horizontal vanishing line at some finite stage: i.e., 
there is $N  \geq 2$ and $h \geq 0$ such that $E_N^{s,t} =0$ for $s>
h$. 
\end{proposition} 

We refer to \cite[Prop. 3.12]{Ma15} for an account of this result in this
language. Of course, the idea is much older, and seems to be elucidated 
in \cite{HPSgeneric}, where the authors consider the more general case of
spectral sequences that have a vanishing line of some (possibly
positive) slope at a finite stage. It appears prominently in the proof  of
many foundational results in chromatic homotopy theory.

Quick convergence indicates that a 
homotopy limit in a stable $\infty$-category which is a priori infinite (such as a totalization)
actually behaves like a finite one, up to taking retracts.\footnote{We remind
the reader that taking retracts, in the $\infty$-categorical setting, is not a
finite homotopy limit.} 
For example, any exact functor preserves finite limits. An exact functor need
\emph{not} preserve totalizations, but it will preserve quickly converging ones
by \Cref{quickconvpreserved}. 

Here is an instance of this phenomenon. 
\begin{proposition} 
\label{quickconvthicksubcat}
Let $X^\bullet \in \fun(\Delta, \mathcal{C})$ 
be a quickly converging cosimplicial object. Then $\mathrm{Tot}(X^\bullet)$
belongs to the thick subcategory of $\mathcal{C}$ generated by the
$\left\{X^i\right\}_{i \in \mathbb{Z}_{\geq 0}}$.
\end{proposition} 
\begin{proof} 
Note that each $\mathrm{Tot}_n (X^\bullet)$ belongs to the thick subcategory
generated by the $\left\{X^i\right\}$ as $\mathrm{Tot}_n$ is a finite homotopy
limit.
It suffices to show that if 
$\left\{Y_i\right\}_{i \geq 0}$ is a quickly converging tower, then
$Y = \varprojlim Y_i$ belongs to the thick subcategory generated by the
$\left\{Y_i\right\}$. In fact, this follows from the claim that $Y$ is a retract
of some $Y_i$. 

To see this, 
we need to produce a map $Y_i \to Y$ for some $i$ such that the composite 
$Y \to Y_i \to Y$ is the identity. 
We consider the functor $\mathcal{C} \to \sp^{op}$ given by $X \mapsto
\hom_{\mathcal{C}}(X, Y)$. 
Since the tower $\left\{Y_i\right\}$ is quickly converging, it follows by
applying this functor (and dualizing) that the map of spectra
\( \varinjlim_i  \hom_{\mathcal{C}}(Y_i, Y) \to \hom_{\mathcal{C}}(Y, Y) \)
is an equivalence, in view of \Cref{quickconvpreserved}. Unwinding the
definitions, it follows that 
there exists $f_i : Y_i \to Y$ such that $Y \to Y_i \stackrel{f_i}{\to} Y$ is
the identity, as desired. 
\end{proof} 

\begin{example} 
Let $G$ be a finite group and let $X \in \fun(BG, \mathcal{C})$. 
We  can form the homotopy fixed points $X^{hG} \in \mathcal{C}$, which can be
recovered as the totalization of a cosimplicial object. 
Namely, choosing the standard simplicial model $EG_\bullet$ of $EG$ as a free
$G$-space, we have
\[  X^{hG} = \mathrm{Map}_G( EG, X) = 
\mathrm{Tot} \left(\mathrm{Map}_G( EG_\bullet, X)\right).
\]
Here $EG_\bullet$ is levelwise a finite free $G$-set; namely, 
$EG_n = G^{n+1}$. Therefore, we have a cosimplicial object
$\widetilde{X}^\bullet$ such that $\widetilde{X}^n = \prod_{G^{n}} X$ such that
$X^{hG} = \mathrm{Tot}(\widetilde{X}^\bullet)$. 

Suppose this cosimplicial object is quickly converging. Then we find: 
\begin{enumerate}
\item $X^{hG}$ belongs to the thick subcategory of $\mathcal{C} $ generated by
$X$. 
\item For any exact functor $F: \mathcal{C} \to \mathcal{D}$, the natural map
$F(X^{hG}) \to F(X)^{hG}$ is an equivalence. 	
\end{enumerate}
In particular, taking homotopy fixed points behaves like a finite homotopy
limit in this case. 
We will return to this example below (see subsection \ref{subsec:nilpaction}) and show the connection to Tate
vanishing. 

\end{example}

Our goal in this subsection is to show that the theory of ``$A$-nilpotence''
is closely connected to the theory of quickly converging towers.
This will take place through the cobar construction. 
First, we will describe the associated tower to the cosimplicial object
$\cb(A)$. 
We now assume that $\mathcal{C}$ is an idempotent-complete, stably
\emph{symmetric monoidal} $\infty$-category. 
Compare \cite[Sec. 4]{MNN15i}. 
\newcommand{\adm}{\mathrm{Adams}}

\newcommand{\const}{\mathrm{Const}}

\begin{cons}
Fix $A \in \alg(\mathcal{C})$. 
We let $I$ be the fiber of $\mathbf{1} \to A$, so that we have a natural map
$I \to \mathbf{1}$. Taking the tensor powers, we obtain a natural tower
$\left\{I^{\otimes k}\right\}_{k \geq 0}$ in $\mathcal{C}$.

Then we define an exact functor
\[  \adm_A :  \mathcal{C} \to \tow(\mathcal{C}) \]
which sends $X \in \mathcal{C}$ to the tower $\adm_A(X) = \left\{I^{\otimes
k} \otimes X \right\}_{ k \geq 0}$, which we call the \textbf{Adams tower}. 
We note that we have a natural map of towers 
\begin{equation} \label{admtoconst} \adm_A(X) \to \const(X),  \end{equation}
where $\const(X)$ denotes the 
constant tower at $X$. 
\end{cons}

We have the following basic result. Although surely folklore, it is a little
tricky to track down in the older literature. We refer to  \cite[Prop.
2.14]{MNN15i} for a proof in modern language. 
\begin{proposition} 
The $\mathrm{Tot}$ tower $\{\mathrm{Tot}_n( \cb(A) \otimes X)\}$ is equivalent
to the cofiber of the map of towers
\eqref{admtoconst}. 
\end{proposition} 

With this in mind, we can state the following basic characterization of
$A$-nilpotence.
It implies \Cref{degeneratefinitevanishingline}, in view of 
\Cref{quickconvss}. We leave the details to the reader. Compare the 
discussion following \Cref{def:nilpexp} below. 

\begin{proposition} 
The following are equivalent for $X \in \mathcal{C}$: 
\begin{enumerate}
\item $X \in \nil_A$. 
\item The cobar construction $\cb(A) \otimes X$ is quickly converging and has
homotopy limit given by $X$. 
\item The Adams tower $\adm_A(X)$ is nilpotent.
\end{enumerate}
\end{proposition}

\begin{example} 
We consider again our basic example of $A = \mathbb{Z}/p\mathbb{Z}$ in
$\mathcal{C} = \md(\mathbb{Z})$. Here the Adams tower is given by 
\[ \dots \to \mathbb{Z} \stackrel{p}{\to} \mathbb{Z}.  \]
We thus recover the fact that $X \in \mathcal{C}$ is $A$-nilpotent if and only
if some power of $p$ annihilates $X$. 
\end{example} 

\subsection{Exponents of nilpotence}

We keep the same notation from the previous subsection. 
We recall that $\nil_A = \tht(A)$ and $\tht(A)$ has a natural \emph{filtration.}
This leads to the following natural definition: 

\begin{definition} 
\label{def:nilpexp}
Fix $X \in \nil_A$.
We will say that $X$ has \textbf{nilpotence exponent $\leq n$} if the following
\emph{equivalent} conditions hold: 
\begin{enumerate}
\item $ X \in \tht(A)_{n}$ in the notation of \Cref{thtfilt}. 
\item The map $I^{\otimes n} \to \mathbf{1}$ becomes null after tensoring
with $X$.
\end{enumerate}
We will write $\exp_A(X)$ for the minimal $n$ such that $X$ has nilpotence
exponent $\leq n$.
\end{definition}

We outline the argument that conditions (1) and (2) are equivalent. 
\begin{enumerate}
\item  
If $X \in \tht(A)_n$, we claim that $I^{\otimes n} \otimes X \to X $ is
null. For $n = 0$, the claim is obvious. 
For $n = 1$, if $X \in \tht(A)_1$, then $X$ is a retract of $A \otimes Y$ for
some $Y$, and the map $I \otimes X \to X$ is nullhomotopic because $ X \to A
\otimes X$ admits a section. In general, we can induct on $n$ and use the
inductive description of $\tht(A)_n$.
\item Suppose $\psi: I^{\otimes n} \otimes X \to X$ is nullhomotopic. 
Then $X$ is a retract of the cofiber of $\psi$. It thus follows that we need to
show that the cofiber of $\psi$ belongs to $\tht(A)_n$. 
It suffices to show that the cofiber of $I^{\otimes n} \to \mathbf{1}$ belongs
to $\tht(A)$, but it is an extension of the $n$ objects 
$\mathrm{cofib}(I^{\otimes (k+1)} \to I^{\otimes k}) 
\simeq A \otimes I^{k}
$ for $k \leq n-1$. 
\end{enumerate}

The exponent of nilpotence is closely related to 
the behavior of the $A$-based Adams spectral sequence. 
We recall that if $X \in \nil_A$, then the $A$-based Adams spectral sequence
collapses at a finite stage with a horizontal vanishing line. 

\begin{proposition}[{Cf. \cite[Prop. 4.4]{Ch98}, \cite[Prop. 2.26]{MNN15ii}}] 
Suppose $X \in \nil_A$ and $\exp_A(X) \leq n$. Then the $A$-based Adams
spectral sequence for $\pi_* \hom_{\mathcal{C}}(Y, X)$ satisfies $E_{n+1}^{s,t}
 = 0$ for $s \geq n$.
\end{proposition}

\subsection{The descent theorem}

We now specialize to the case where \emph{every} object is $A$-nilpotent.

\begin{definition}[Balmer \cite{Balmersep}, Mathew \cite{MGal} ] 
We say that the algebra object $A$ is \textbf{descendable} or of
\textbf{universal descent} if 
$\mathbf{1} \in \nil_A = \tht(A)$ (equivalently, if $\mathcal{C} = \tht(A)$).
\end{definition}

When $A$ is descendable, many properties of objects or morphisms in
$\mathcal{C}$ can be checked after tensoring up to $A$, sometimes up to
nilpotence.

\begin{proposition} 
\label{descisconservative}
Suppose $A$ is descendable. Then if $X \otimes A = 0$, then $X = 0$. 
\end{proposition} 
\begin{proof} 
The class of all $Y$ such that $X \otimes Y = 0$ is clearly a thick
$\otimes$-ideal; if it contains $A$, it must therefore contain $\mathbf{1}$, so
that $X = 0$. 
\end{proof} 

As another example, we have the following proposition, which can be proved using a
diagram-chase. 
We leave the details to the reader (or compare the discussion in \cite[Sec.
4]{MNN15i} and \cite[Sec. 3]{MGal}).
\begin{proposition} 
Say that a map $f: X \to Y$ in $\mathcal{C}$ is \textbf{$A$-zero} if
$1_A \otimes f: A \otimes X \to A \otimes Y$ is nullhomotopic. 
The following are equivalent: 
\begin{enumerate}
\item $A$ is descendable.  
\item If $I = \mathrm{fib}( \mathbf{1} \to A)$, then the map $I \to \mathbf{1}$ is
$\otimes$-nilpotent, i.e., there exists $N$ such that $I^{\otimes N} \to
\mathbf{1}$ is null. 
\item There exists $N \geq 1$ such that if $f_1, \dots, f_N$ are composable $A$-zero
maps, then $f_N \circ \dots \circ f_1 = 0$. 
\end{enumerate}
\end{proposition}

\begin{proposition} 
\label{nilpdescnec}
Let $A \in \alg(\mathcal{C})$ be descendable. Let $R$ be an algebra object in
$\mathcal{C}$. Suppose $x \in \pi_*R $ maps to zero in $\pi_*(A \otimes
R)$. Then $x$ is nilpotent. 
\end{proposition} 
\begin{proof} 
This is a special case of the previous result. 
Alternatively, we can use the $A$-based Adams spectral sequence to compute
$\pi_* R$. Since the spectral sequence degenerates with a horizontal vanishing
line at a finite stage, we see easily that any permanent cycle in positive
filtration must be nilpotent. 
\end{proof}

\begin{example} 
Suppose that $\mathbf{1}$ is a retract of $A$. 
Then clearly $A$ is descendable, and we can 
take $N = 1$ in the above. This case was considered by Balmer
\cite{Balmerdesc}, who shows that  a type of descent at the level of homotopy
(triangulated) categories holds. When $A$ is descendable but $\mathbf{1}$ is
not a retract, one still has a descent statement, but it requires the use of
$\infty$-categories which we explain below. 
\end{example}

One of the key features of the notion of descendability is that it enables a formulation of a
derived analog of faithfully flat descent, as in \cite{MGal}.  
This essentially uses the theory of $\infty$-categories. 
We refer also
to \cite[Appendix D.3]{lurie_sag} for another account of these results.  

If $A \in \mathrm{Alg}(\mathcal{C})$, then we can form a stable $\infty$-category 
of $A$-modules $\md_{\mathcal{C}}(A)$. 
If $\mathcal{C}$ is presentable and $A$ is a \emph{commutative} algebra in
$\mathcal{C}$, then $\md_{\mathcal{C}}(A)$ acquires a symmetric monoidal
structure from the $A$-linear tensor product. 
We refer to \cite{Lur16} for an $\infty$-categorical treatment of these
ideas, and the original source \cite{EKMM} that developed the theory in
spectra. 

The basic result is that one can recover $\mathcal{C}$ as the
$\infty$-category of ``$A$-modules with descent data,'' if $A$ is descendable.
This is an $\infty$-categorical version of the classical situation in algebra,
which we review briefly.

\begin{example} 
\label{classicalflatdesc}
We remind the reader of the classical setting in ordinary algebra
as developed in
\cite[Exp. VIII]{SGA1}. A modern exposition is in \cite{Vistoli}.  
Let $R \to R'$ be a map of commutative rings. 
Given an $R$-module $M$, we can form the base-change $M_{R'} = R' \otimes_R M$, which
is an $R'$-module. 
The $R'$-module $M_{R'}$ comes with the following extra piece of data: there is
an isomorphism of $R' \otimes_R R'$-modules
\[ \phi_M:  R' \otimes_R M_{R'} \simeq M_{R'} \otimes_R R'  \]
(coming from the fact that both are base-changed from $M$ along $R \to R'
\otimes_R R'$). 
In addition, $\phi_M$ satisfies a cocycle condition. 
Namely, the map $\phi_M$ yields \emph{two} natural isomorphisms of $R' \otimes_R R'
\otimes_R R'$-modules
\[ R' \otimes_R R' \otimes_R M \simeq M \otimes_R R' \otimes_R R '  \]
given by $(\phi_{M, 12} \otimes 1_{R'}) \circ (1_{R'} \otimes \phi_{M, 23})$
and $\phi_{M, 13}$. 
In this case, we have an equality of maps 
$\phi_{M, 13} = 
(\phi_{M, 12} \otimes 1_{R'}) \circ (1_{R'} \otimes \phi_{M, 23})$; this is
called the \emph{cocycle condition}. 

In general, an \emph{$R'$-module with descent datum} consists of an $R'$-module $N'$ and an
isomorphism $\phi: R' \otimes_R N \simeq N \otimes_R R'$ of $R' \otimes_R
R'$-modules satisfying the cocycle condition. 
Any $R$-module yields in the above fashion an $R'$-module with descent datum.
Grothendieck's faithfully flat descent theorem states that, if $R \to R'$ is
faithfully flat, this functor
implements an equivalence of categories between $R$-modules and $R'$-modules
with descent datum.
\end{example} 

\begin{example} 
Consider the special case $\mathbb{R} \to \mathbb{C}$. 
In this case, given an $\mathbb{R}$-vector space $V$, 
we can form the complexification $V_{\mathbb{C}} = \mathbb{C}
\otimes_{\mathbb{R}} V$. The $\mathbb{C}$-vector space $V_{\mathbb{C}}$ is
equipped with a $\mathbb{C}$-antilinear involution $\iota: V_{\mathbb{C}} \to
V_{\mathbb{C}}$ given by complex conjugation on the $\mathbb{C}$ factor. 
In this case, faithfully flat descent states that the category of
$\mathbb{R}$-vector
spaces
is equivalent to the category of $\mathbb{C}$-vector spaces equipped with an
antilinear involution. 
Note that we can recover $V$ as the fixed points of the $C_2$-action on
$V_{\mathbb{C}}$.
More generally, given a $G$-Galois extension of fields $K \subset L$, a
descent datum on an $L$-vector space $W$ consists of a $G$-action on $W$ which
is $L$-\emph{semilinear}, i.e., for a scalar $l \in L$ and $w \in W$, we have
$g( lw) = g(l) g(w)$ for $g \in G$ and $g(l) \in L$ comes from the natural action of
$g$ on $L$. 
\end{example}

In the $\infty$-categorical setting, 
it is generally unwieldy to spell out explicitly the higher analogs of the coherence
condition (which are replaced by with higher homotopies). 
However, it is possible  to describe an $\infty$-categorical analog of the
category of descent data as an appropriate homotopy limit, or as an $\infty$-category
of coalgebras. 
We refer also to the work of Hess \cite{Hess} for a detailed treatment in a
related setting.

Let $\mathcal{C}$ 
be presentably symmetric monoidal stable $\infty$-category. 
As mentioned earlier, for every $A \in \clg(\mathcal{C})$, we have associated a
presentably symmetric monoidal $\infty$-category $\md_{\mathcal{C}}(A)$. 
This is \emph{functorial} in $A$. Let $\cati$ be the $\infty$-category of
$\infty$-categories. Then we have a functor 
$\clg(\mathcal{C}) \to \cati$ sending $A \mapsto \md_{\mathcal{C}}(A)$ and a
map $A \to A'$ to the relative tensor product functor $A' \otimes_A \cdot :
\md_{\mathcal{C}}(A) \to \md_{\mathcal{C}}(A')$ .

If $A \in \clg(\mathcal{C})$, then we can form the cobar construction 
$\cb(A)$ as a diagram in $\fun(\Delta, \clg(\mathcal{C}))$, i.e., as a
cosimplicial diagram in $\mathcal{C}$. Similarly, the augmented cobar
construction $\cbaug(A)$ is a diagram in $\clg(\mathcal{C})$.

\newcommand{\desc}{\mathrm{Desc}}

\begin{definition} 
Let $\mathcal{C}$ 
be a presentably symmetric monoidal stable $\infty$-category and let $A \in
\clg(\mathcal{C})$. The \textbf{$\infty$-category of descent data}
$\desc_A(\mathcal{C})$ is given by
the totalization 
in $\cati$ 
\begin{equation} \desc_A(\mathcal{C}) \stackrel{\mathrm{def}}{=}
\mathrm{Tot}( \md_{\mathcal{C}}(\cb(A))) = 
\mathrm{Tot}\left( \md_{\mathcal{C}}(A) \rightrightarrows
\md_{\mathcal{C}}(A \otimes A) \triplearrows \dots \right)  \end{equation}
Note that since $\cb(A)$ receives an augmentation from $\mathbf{1}$, we obtain
a natural symmetric monoidal, cocontinuous functor 
\begin{equation} \label{compdesc}\mathcal{C} \to \desc_A(\mathcal{C}).\end{equation}
\end{definition} 

\begin{remark} 
Informally, an object of $\desc_A(\mathcal{C})$ consists of the following: 
\begin{enumerate}
\item For each $[n] \in \Delta$, an $A^{\otimes (n+1)}$-module $M_{[n]}$.  
\item 
For each map $[n] \to [m]$ in $\Delta$, an equivalence
$ M_{[m]} \simeq A^{\otimes (m+1)}  \otimes_{A^{\otimes (n+1)}} M_{[n]}$. 
\item Cocycle conditions for these equivalences, and higher homotopies.
\end{enumerate}
\end{remark} 

This yields a generalization of 
the construction of \Cref{classicalflatdesc}. 
Note that the approach to descent data considered in \cite{Hess}
is based on comonadicity, which applies in more general situations than this. 
In this setting, a \emph{descent theorem} will state that a comparison map of
the form \eqref{compdesc} is an equivalence.

There is a direct analog of faithfully flat descent when we work in the case
$\mathcal{C} = \sp$ or $\md_{}(A)$. 
\begin{definition} 
A map of $\e{\infty}$-rings $A \to B$ is said to be \emph{faithfully flat} if: 
\begin{enumerate}
\item $\pi_0 A \to \pi_0 B$ is faithfully flat.  
\item The natural map $\pi_0 B \otimes_{\pi_0 A} \pi_* A \to \pi_* B $ is an
isomorphism. 
\end{enumerate}
\end{definition} 

In this case, one can prove an analog of faithfully flat descent for modules.
\begin{theorem}[{Lurie \cite[Theorem D.6.3.1]{lurie_sag}, \cite[Theorem
6.1]{DAGVII}} ] 
Suppose $A \to B$ is a faithfully flat map of $\e{\infty}$-rings.
Then the natural map implements an equivalence of symmetric monoidal
$\infty$-categories
\( \md_{}(A) \simeq \desc_B( \md_{}(A)).   \)
\end{theorem}

A key result in the theory of ``descent up to nilpotence''
is that the analogous conclusion holds when one assumes descendability. 
Unlike faithful flatness, descendability is a purely categorical condition,
which makes sense in any stably symmetric monoidal $\infty$-category (not only
$\sp$). 

\begin{theorem}[{\cite[Prop. 3.22]{MGal}}] 
\label{nilpdescthm}
Suppose $A \in \clg(\mathcal{C})$ is descendable.
Then the natural functor implements an equivalence
of symmetric monoidal $\infty$-categories
\[ \mathcal{C} \simeq \desc_A(\mathcal{C}). \]
\end{theorem}

\begin{proof}[Proof sketch] 

The descent theorem is an application of the $\infty$-categorical version of
the Barr-Beck monadicity theorem proved in \cite{Lur16} (see also \cite{RV}
for another approach to the monadicity theorem). 
As in \cite[Lemma D.3.5.7]{lurie_sag}, one can identify the $\infty$-category
$\desc_A(\mathcal{C})$ with the $\infty$-category of \emph{coalgebras} in
$\md_{A}(\mathcal{C})$ over the comonad arising from the adjunction
$\mathcal{C} \rightleftarrows \md_{\mathcal{C}}(A)$ given by extension and
restriction of scalars. 
The key point is then to show that the adjunction $\mathcal{C} \rightleftarrows
\md_A(\mathcal{C})$ is comonadic. 
By the monadicity theorem, we need to show two statements: 
\begin{enumerate}
\item  The functor $A \otimes \cdot: \mathcal{C} \to \md_{\mathcal{C}}(A)$ is
conservative.
\item The functor $A \otimes \cdot: \mathcal{C} \to \md_{\mathcal{C}}(A)$
preserves totalizations in $\mathcal{C}$ which admit a splitting
after tensoring with $A$. \end{enumerate}
Statement (1) follows from \Cref{descisconservative}. Statement (2) follows in
a similar (but more elaborate) fashion: 
given any $X^\bullet \in \fun(\Delta, \mathcal{C})$ such that $A \otimes
X^\bullet$ is quickly converging (e.g., admits a splitting), we argue that $X^\bullet$ is quickly converging. This means that any exact functor, e.g., tensoring with $A$, preserves the inverse limit of
$X^\bullet$ by \Cref{quickconvpreserved}. 
\end{proof} 

\section{Examples of nilpotence}

\subsection{First examples of descendability}
We start by giving a number of elementary examples, mostly from \cite{MGal}. 

Let $A \in \clg(\mathcal{C})$. 
Then an $A$-module is nilpotent if it belongs to the thick subcategory of
$\mathcal{C}$ generated by those objects which admit the structure of
$A$-module; in particular, $A$ is descendable if the unit has this property.
Sometimes, we can check this directly.

\begin{example} 
\label{quotbynilpideal}
Let $R$ be a commutative ring and let $I \subset R$ be a nilpotent ideal. 
Then the map $R \to R/I$ is clearly descendable. 
In fact, the  (finite) $I$-adic filtration on $R$  gives a filtration of $R$
with subquotients $R/I$-modules. 
\end{example} 

\begin{example} 
Let $R$ be a connective $\e{\infty}$-ring which is $n$-truncated, i.e., such
that $\pi_i R = 0$ for $i > n$. Then the map $R \to \pi_0 R = \tau_{\leq 0} R$
is descendable for a similar reason: we have the Postnikov filtration of $R$
whose subquotients admit the structure of $\pi_0 R$-modules (internal to
$\md(R)$). 
\end{example}

\begin{example} 
Suppose $\mathbf{1} \in \clg(\mathcal{C})$ is a \emph{finite} inverse limit of
$A_\alpha \in \clg(\mathcal{C}), \alpha \in I$. Then $\prod_{\alpha} A_\alpha$ is
descendable. 
For example, if an $\e{\infty}$-ring $R$ is a finite inverse limit of
$\e{\infty}$-rings $R_\alpha$, then the map $R \to \prod R_\alpha$ is
descendable. 
\end{example} 

\begin{example} 
Suppose $R$ is an $\e{\infty}$-ring and $X$ is a finite, pointed, connected CW complex. Then the map 
$R^X = C^*(X; R) \to R$ given by evaluation at the basepoint is descendable. 
\end{example}

\begin{example} 
Suppose $R$ is an $\e{\infty}$-ring with $\pi_0 R = k$ a field and suppose
given a map $R \to k$ inducing the identity on $\pi_0$. 
Suppose that $\pi_i R = 0$ for $i > 0$ and for $i \ll 0$. 
Then the map $R \to k$ is descendable. 
This follows from the work of \cite{DGI}. 

In fact, we claim that if $M$ is an $R$-module such that $\pi_i M = 0$ for $i
\notin [a,b]$ for some $a \leq b \in \mathbb{Z}$, then $M$ is $k$-nilpotent;
applying this to $R$ itself we can conclude. 
Using \cite[Prop. 3.3]{DGI}, we can induct on $b-a$ and reduce to the case 
where $a = b$; then \cite[Prop. 3.9]{DGI} implies that $M$ is actually a
$k$-module inside of $\md(R)$. 
\end{example}

One necessary (but not sufficient) condition for descendability of $A \in \clg(\mathcal{C})$ is that the kernel of the map 
$\pi_* \mathbf{1} \to \pi_* A$ 
consists of \emph{nilpotent} elements by \Cref{nilpdescnec}. In fact, an
elaboration of that argument shows that the kernel of the
aforementioned map needs to be a nilpotent ideal in $\pi_* \mathbf{1}$. 
Another necessary condition is that 
descent for modules holds. The latter condition is again insufficient in view
of the following example.

\begin{example} 
Let $\mathcal{C} = \widehat{\md(\mathbb{Z})}_p$ denote the $\infty$-category of
$p$-complete objects in $\md(\mathbb{Z}) \simeq D(\mathbb{Z})$; this is
equipped with the $p$-adically completed tensor product. Let $A =
\mathbb{Z}/p\mathbb{Z}$. Then tensoring with $A$ is conservative. Since $A$ is
dualizable, tensoring with $A$ commutes with all limits. By the
$\infty$-categorical Barr-Beck
theorem,
the natural map $\mathcal{C} \to \desc_{\mathcal{C}}(A)$ is an equivalence
(cf. the discussion in \cite[Lemma D.3.5.7]{lurie_sag}). However, $\mathbb{Z}/p\mathbb{Z}$ is clearly
not descendable.  

\end{example}

Nonetheless, under strong conditions descendability is actually equivalent to
this conclusion. 
The following appears as \cite[Th. 3.38]{MGal}; however, the proof given there
uses more abstract category theory than necessary. We give a simplified
presentation here. 

\begin{theorem} 
\label{dualdesc}
Let $B \in \alg(\mathcal{C})$. Suppose $\mathbf{1}$ is compact and $B$ is
dualizable. Suppose that $\cdot \otimes B: \mathcal{C} \to \mathcal{C}$ is conservative. Then $B$ is descendable. 
\end{theorem} 
\begin{proof} 
Note first that tensoring with $\mathbb{D} B$ is also conservative.
In fact, this follows because 
$B$ is a retract of $B \otimes \mathbb{D}B \otimes B$ because $B$ is a module
over $B \otimes \mathbb{D} B$. Note that $B$ and $\mathbb{D}B $ generate the
same thick $\otimes$-ideal in $\mathcal{C}$. 

We consider the augmented cobar construction $\cbaug(B) \in \fun(\Delta^+,
\mathcal{C})$. It takes values in dualizable objects; therefore, we consider
the dual, $\mathbb{D}\cbaug(B) \in \fun(\Delta^{+, op}, \mathcal{C})$. 
Since $\cbaug(B) \otimes B$ is split, it follows that 
$\mathbb{D}\cbaug(B) \otimes \mathbb{D} B$ is a split augmented simplicial
diagram and is therefore a colimit diagram.  
Since tensoring with $\mathbb{D} B$ is conservative, it follows that
$\mathbb{D} \cbaug(B)$ is a colimit diagram. 
It follows that the natural map
\[ \left\lvert \mathbb{D} \cb(B)\right\rvert \to \mathbf{1}  \]
is an equivalence in $\mathcal{C}$. The geometric realization is 
the filtered colimit of its finite skeleta, 
$\left\lvert \mathbb{D} \cb(B)\right\rvert = \varinjlim \left\lvert
\mathrm{sk}_n \mathbb{D} \cb(B)\right\rvert
$. Since $\mathbf{1}$ is compact, it follows that $\mathbf{1} $ is a retract of 
$\mathrm{sk}_n \mathbb{D} \cb(B)$ for some $n$.  Therefore, $\mathbf{1}$ is
$B$-nilpotent and $B$ is descendable. 
\end{proof} 

Even under compact generation, it is far from sufficient 
to assume that tensoring with an object is conservative to guarantee
descendability. We give a simple example below.

\begin{example} 
Let $\mathcal{C} = \md(\mathbb{Z}_{(p)})$. Then $A = \mathbb{Q} \times \mathbb{F}_p$
has the property that tensoring with $A$ is conservative on $\mathcal{C}$, but
$A$ is not descendable. 
For example, the exact functor $\md(\mathbb{Z}_{(p)}) \to
\md(\mathbb{Z}_{(p)})$ given by $X \mapsto
(\widehat{X}_p)_{\mathbb{Q}}$ annihilates any $A$-module but is nonzero. 
\end{example} 

\subsection{Maps of discrete rings}

Given the motivation of faithfully flat descent, it is a natural question now to ask when 
a map of discrete rings is descendable.

\begin{question} 
\label{ffq}
Let $f: A \to B$ be a faithfully flat morphism of commutative rings. 
Is $f$ descendable? 
\end{question} 

In general, we do not know the answer to the above question. We can restate it
as follows. 
Consider the homotopy fiber sequence
\begin{equation} \label{basicfiber} I \xrightarrow{\phi} A \to B  \end{equation}
in the derived $\infty$-category $D(A)$. 
Then $I = M[-1]$ where $M$ is an $A$-module 
which is flat; namely, $M = B/A$. 
It suffices to show that $\phi$ is $\otimes$-nilpotent, i.e., for some $n$, the
map 
\( \phi^{\otimes n}: I^{\otimes n} \to A   \)
is nullhomotopic in $D(A)$.
Now $\phi: I \to A$ is an example of a \emph{phantom} map in $D(A)$. 
We recall the definition below. 

\begin{definition} 
A map $N \to N'$ in $D(A)$ is said to be \textbf{phantom} if for every
perfect $A$-module $F$ with a map $F \to N$, the composite $F \to N \to N'$ is
nullhomotopic. 
\end{definition}

We refer to \cite{ChSt, Neeman97} for a treatment of phantom maps in some
generality. 
Equivalently, a phantom map is a filtered colimit of maps, each of which is
nullhomotopic (but the nullhomotopies need not be compatible with the filtered
colimit). 
In general, in \eqref{basicfiber}, $I$ is a filtered colimit $I = \varinjlim
M_\alpha[-1]$, where each $M_\alpha$ is a  
finitely generated free $A$-module, by Lazard's theorem (see, e.g., \cite[Tag
058G]{Stacks}) and the induced map $M_\alpha[-1] \to M[-1] \to A$ is clearly
nullhomotopic, which means that $\phi$ is phantom as desired. 

In general, phantom maps are very far from being zero, but under countability
conditions the class of phantom maps can be shown to be a square-zero ideal. 
That is the content of the following classical result. 

\begin{theorem}[Christensen-Strickland \cite{ChSt}, Neeman \cite{Neeman97}] 
Let $A$ be a countable ring. Then the composite of any two phantom maps in
$D(A)$ is nullhomotopic. 
\end{theorem} 

The result is based on showing that homology theories in $D(A)$ satisfy Brown
representability.
It follows that if $A$ is a countable ring and $A \to B$ is faithfully flat,
then $B$ is descendable: in fact, $A$ is $B$-nilpotent of $B$-exponent at most
two because $\phi^{\otimes 2}$ is nullhomotopic. 
Using recent work of Muro-Ravent{\'o}s \cite{MuR}, one can extend this to show that if $A$ has
cardinality at most $\aleph_k$ for some $k < \infty$, then the composite of
$k+2$ phantom maps is zero. 
This leads to the following result:

\begin{proposition}[{\cite[Prop. 3.32]{MGal} or \cite[Prop. D.3.3.1]{lurie_sag}}] 
Suppose the cardinality of $A$ is at most $\aleph_k$ for some $k \in [0,
\infty)$. Then the answer to \Cref{ffq} is positive.
\end{proposition}

In general, 
we do not know whether Question~\ref{ffq} holds without cardinality hypotheses.
Nonetheless, the answer is positive if one assumes that $A$ is
\emph{noetherian} of 
finite Krull dimension. 
We are indebted to Srikanth Iyengar for pointing out the following to us. 
\begin{theorem} 
\label{thm:flatdescsri}
Suppose $A$ is a
noetherian ring of finite Krull dimension and suppose $B$ is a faithfully flat
$A$-algebra. Then $B$ is descendable. 
\end{theorem} 

\begin{proof} 
Keep the same notation as in \eqref{basicfiber}.
 Choose $n > \dim A$. Then $I^{\otimes n} \simeq M^{\otimes
n}[-n]$, so that the map $\phi^{\otimes n}$ is classified by an element in
$\mathrm{Ext}^n_A( M^{\otimes n}, A)$. We claim that this group itself
vanishes, so that $\phi^{\otimes n}$ is nullhomotopic. 

We use \cite[Cor. 7.2]{GJ} to observe now that, as a flat $A$-module, the
projective dimension of $M^{\otimes n}$ is at most $\dim A$. It follows that
the group 
$\mathrm{Ext}^n_A( M^{\otimes n}, A)$ itself vanishes, which proves that $\phi$
is $\otimes$-nilpotent as desired. 
\end{proof}

\begin{example} 
Let $(A, \mathfrak{m})$ be a noetherian local ring. Then the map $A \to
\hat{A}$ is descendable. 
\end{example} 

If we do not assume finite Krull dimension then, by contrast, we do not even
know if $A \to \prod_{\mathfrak{p} \in \spec A}
A_{\mathfrak{p}}$ (which is faithfully flat if $A$ is noetherian) is descendable. 

We can also easily extend \Cref{thm:flatdescsri} to the case of ring spectra. A
map $A \to B$ of $\e{\infty}$-ring spectra is said to be \emph{faithfully flat}
if $\pi_0 A \to \pi_0 B$ is faithfully flat and the natural map $\pi_* A
\otimes_{\pi_0 A } \pi_0 B \to \pi_* B$ is an isomorphism. 

\begin{corollary} 
Suppose $A$ is an $\e{\infty}$-ring with $\pi_0 A$ noetherian of finite Krull
dimension and $B$ is a
faithfully flat $\e{\infty}$-$A$-algebra. Then $A \to B$ is descendable.
\end{corollary} 
\begin{proof} 
Again, we form the fiber sequence $I \stackrel{\phi}{\to} A \to B$. Let $n  > \dim A$; then we
argue that $I^{\otimes n} \to A$ is nullhomotopic. 
In general, given $N_1, N_2 \in \md(A)$, we have the usual spectral sequence
(cf. \cite[Ch. IV, Th. 4.1]{EKMM})
\[ E_2^{s,t} = \mathrm{Ext}^{s,t}_{\pi_*(A)}(\pi_*(N_1), \pi_*(N_2)) \implies
\pi_{t-s} \hom_{A}(N_1, N_2).  \]
Note that $\phi$ induces zero on homotopy, so it is detected in filtration at
least one in the $\mathrm{Ext}$ spectral sequence. It follows that 
$\phi^{\otimes n}$ is detected in filtration at least $n$ in the $\mathrm{Ext}$
spectral sequence. 
However, we see already (using \cite[Cor. 7.2]{GJ}) that everything in
filtration $\geq n$ of the $\mathrm{Ext}$ spectral sequence for
$\hom_A(I^{\otimes n}, A)$ vanishes, so that $\phi^{\otimes n}$ must vanish.
\end{proof} 

We have already seen some examples
where a map which is far from faithfully flat is nevertheless descendable, for
example the quotient by a nilpotent ideal. 
Here is another example. We refer to \cite[Prop. 5.25]{BS} for a more general
statement, showing that the condition of descendability in the noetherian case is equivalent 
to being a cover in the \emph{$h$-topology.}
As this special case is 
elementary, we illustrate it here. 

\begin{theorem} 
Let $f: A \to B$ be a finite map of noetherian rings such that $\mathrm{ker}
f$ is nilpotent (equivalently, $\spec B \to \spec A$ is surjective). Then  $f$ is descendable. 
\end{theorem}
In this proof, we will break with convention slightly and indicate derived
tensor products by $\stackrel{\mathbb{L}}{\otimes}$. This is to avoid
confusion, as we will also work with strict quotients by ideals. 
\begin{proof} 
Let $Z \subset \spec A$ be a closed subscheme, defined by an ideal $J \subset
A$. 
We consider the condition that $A/J \to B/JB$ is descendable.
Using \Cref{quotbynilpideal}, one sees that this only depends on the radical of
$J$. Thus we can consider the condition 
on closed \emph{subsets} of $\spec A$. 

By noetherian induction, we may assume that the condition holds for all proper subsets of
$\spec A$, i.e., that if $J$ is a non-nilpotent ideal, then $A/J \to B/JB$ is
descendable.
We then need to prove that $A \to B$ is descendable. 

First assume that $\spec A$ is reducible.
Let $\{\mathfrak{p}_i\}$ be the  minimal prime ideals of $A$. Then the map $A \to
\prod_{i} A/\mathfrak{p}_i$ is descendable since the intersection of the
$\mathfrak{p}_i$ is nilpotent.
By hypothesis, $A/\mathfrak{p}_i \to B/\mathfrak{p}_iB $ 
is descendable because no $\mathfrak{p}_i$ is nilpotent.
However, now by a transitivity argument it follows that $A \to B$ is descendable.

Thus, we may assume that $\spec A$ is irreducible, so that it has a unique
minimal prime ideal $\mathfrak{p}$, which is nilpotent.  Replacing $A, B$ by
$A/\mathfrak{p}, B/\mathfrak{p}B$, we may reduce to the case where
$\mathfrak{p} = 0$ so that $A$ is a domain. 
This forces $A \to B$ to be injective. 
Let $K(A)$ be the quotient field of $A$. 
Since $B \otimes_{A} K(A)$ is finite free over $K(A)$, 
there exists a nonzero $f \in A$ such that $B_f$ is
finite free over $A_f$, so that $A_f \to B_f$ is clearly descendable. 

In particular, the map $\phi$ in \eqref{basicfiber} 
gives an element of $\mathrm{Ext}^1_A(B/A, A)$ which is $f$-power torsion. 
Suppose $f^N \phi = 0$.
Note that $A/(f) \to B/(f)$ is descendable by the inductive hypothesis; therefore, so is 
$A/f \to B \stackrel{\mathbb{L}}{\otimes}_A A/f$. 
It follows that $\phi^{\otimes n}$ becomes nullhomotopic after (derived) base-change
along $A \to A/f$ for some $n$. 
Consider the object $\hom_{D(A)}( B/A^{\mathbb{L} \otimes n}, A) \in D(A)$ in which 
$\phi^{\otimes n}$ lives in $\pi_0$.  
Since $B/A$ is a finitely generated $A$-module, we see easily that 
we have an equivalence in $D(A/f)$,
$$\hom_{D(A)}( B/A^{\mathbb{L} \otimes n}, A) 
\stackrel{\mathbb{L}}{\otimes}_A A/f =
\hom_{D(A/f)}( B/A^{\mathbb{L} \otimes n} \stackrel{\mathbb{L}}{\otimes} A/f, A/f) 
.$$

Let $N \in D(A)$. The long exact sequence associated to the cofiber sequence
\( N \stackrel{f}{\to} N \to N \stackrel{\mathbb{L}}{\otimes}_A A/f   \)
shows that if $x \in \pi_0 N$ maps to zero in 
$N \stackrel{\mathbb{L}}{\otimes}_A A/f$, then $x$ is divisible by $f$. 
Taking $N = \hom_{D(A)}( B/A^{\mathbb{L} \otimes n}, A) $, we see that  we have
$\phi^{\otimes n} = f \psi$ for some $\psi \in \pi_0 \hom_{D(A)}(
B/A^{\mathbb{L} \otimes n}, A)$. It now follows that as $f^{N}
\phi = 0$, we have $\phi^{\otimes n N+1} = 0$. This proves that $A \to B$ is
descendable. 
\end{proof} 

We refer to \cite[Sec. 5]{BS} for some applications of these ideas in the
setting of \emph{perfect} rings, and in particular $h$-descent for
quasi-coherent sheaves.

\subsection{Galois extensions}

In this subsection, we note another example of 
descendability: the \emph{Galois extensions} studied by Rognes \cite{Rog08}. 
We begin with an important special example.
\begin{proposition} 
The map 
$KO \to KU$ is descendable. 
\end{proposition} 

\begin{proof} 
We use the basic equivalence, due to Wood, of $KO$-module spectra
\begin{equation} \label{KUKO}  KU \simeq KO   \wedge C \eta,  \end{equation}
where $C \eta \simeq \Sigma^{-2} \mathbb{CP}^2$ denotes the cofiber of the Hopf map, and now the fact that
$\eta$ is nilpotent. 
Namely, the equivalence
\eqref{KUKO}
implies inductively by taking cofiber sequences that $KO \wedge C \eta^{n}$ is $KU$-nilpotent for $n \geq
1$. 
Taking $n \geq 4$ (so that $\eta^4 =0 $),\footnote{Or $n \geq 3$, since $\eta^3
= 0$ in $\pi_*(KO)$.} we find that $KO$ itself is $KU$-nilpotent, as desired.
We  note that this argument is very classical and appears, for example, in the
proof of \cite[Cor. 4.7]{Bou79}. 
\end{proof} 

\newcommand{\tmf}{\mathrm{tmf}}
The analog also works for the map of connective spectra $ko \to ku$, because
the equivalence \eqref{KUKO} passes to the connective versions as well. 

We note a basic phenomenon here: the homotopy groups of $KU$ are very simple
homologically, i.e., $\pi_*(KU) \simeq \mathbb{Z}[\beta^{\pm 1}]$. By contrast,
the homotopy groups of $KO$ have infinite homological dimension, $\pi_*(KO)
\simeq \mathbb{Z}[\eta, t_4, u_8^{\pm 1}]/(2 \eta, \eta^3, t_4^2  = 2u_8)$. 
As a result, it is generally much easier to work with $KU$-modules than with $KO$-modules. 
Note that this simplification of the homological dimension is very different
from the setting of faithfully flat descent.

There are more complicated versions of this. 
For example, we have a map of $\e{\infty}$-rings $\tmf_{(2)} \to \tmf_1(3)_{(2)}$.
Here $\tmf_{(2)}$ has extremely complicated homotopy groups, while
$\pi_* \tmf_1(3)_{(2)} \simeq \mathbb{Z}_{(2)}[v_1, v_2]$  as in \cite{LN12}.  
In \cite{Htmf}, it is shown that there is a 2-local finite complex $DA(1)$ such that
$\tmf_{(2)} \wedge DA(1)  \simeq \tmf_1(3)_{(2)}$, a result originally due to
Hopkins-Mahowald. Here $DA(1)$ has torsion-free homology, and by the thick
subcategory theorem \cite{HS98}, the thick subcategory that $DA(1)$ generates
in spectra contains the 2-local sphere. 
Therefore, $\tmf_{(2)} \to \tmf_1(3)_{(2)}$ is easily seen to be
descendable.

\begin{definition}[Rognes \cite{Rog08}] 
Let $G$ be a finite group.
A \textbf{faithful} $G$-\textbf{Galois extension} 
of an $\e{\infty}$-ring $A$ consists of an $\e{\infty}$-$A$-algebra $B$
together with a $G$-action such that:
\begin{enumerate}
\item The natural map $A\to B^{hG}$ is an equivalence. 
\item  The natural map $B \otimes_A B \to \prod_G B$ is an equivalence.
\item $\otimes A$ is conservative on $\mathcal{C}$. 
\end{enumerate}
\end{definition}

\begin{example} 
Let $A \to A'$ 
be a $G$-Galois extension of commutative rings (i.e., $\spec A' \to \spec A$ is
a $G$-torsor in the sense of algebraic geometry). Then the induced map on
associated $\e{\infty}$-rings is Galois in Rognes's sense.
\end{example}

\begin{example} 
The fundamental example of a Galois extension is that $KO \to KU$ is a faithful
$C_2$-Galois extension, where $C_2$ acts on $KU$ via complex conjugation. 
The extensions of connective spectra $\tmf_{(2)} \to \to \tmf_1(3)_{(2)}$ are
not Galois, but the periodic analogs fit into a Galois picture. Compare
\cite{MM15}. 
\end{example} 

\begin{proposition} 
\label{galisdesc}
Suppose 
$A \to B$ is a faithful $G$-Galois extension. Then $A \to B$ is descendable.
\end{proposition} 
\begin{proof} 
In fact, $B$ is dualizable as an $A$-module by \cite[Prop. 6.2.1]{Rog08}.
Since tensoring with $B$ is conservative on $\md(A)$, 
the result now follows from \Cref{dualdesc}. 
\end{proof} 

We refer to \cite{MGal} for more details on the relationship between Galois
theory and descendability. In particular, a treatment is given there of how one
can fit Rognes's theory into the ``axiomatic Galois theory'' of \cite{SGA1}:
the Galois extensions are precisely torsors with respect to descendable
morphisms. 

We end this subsection by describing how the general descent theorem
(\Cref{nilpdescthm})
applies in the case of a Galois extension. If $A \to B$ is $G$-Galois, then the
$G$-action on $B$ induces a $G$-action on the $\infty$-category $\md(B)$. 
The following result has been observed independently by many authors. Compare in
particular \cite{GL, Hess09,
Mei12, Ban13}. \begin{theorem}[Galois descent] 
Suppose $A \to B$ is a faithful $G$-Galois extension. Then we have an
equivalence of symmetric monoidal $\infty$-categories $\md(A) \simeq
\md(B)^{hG}$.  
\end{theorem} 

This result can also be proved for $G$-Galois extensions when $G$ is not
necessarily finite (which we do not discuss here). We refer to \cite[Th. 9.4]{MGal} for the statement when $G$
is allowed to be a topological group and \cite{MPic} for an application. 

\begin{remark} 
We refer to \cite{MS, HMS} for 
a description of how these results can be used to calculate Picard groups of
ring spectra. For instance, it is  
possible to calculate the Picard group of $KU$-modules relatively directly
using the homotopy groups of $KU$ (compare \cite{BakerRichter}), while
invertible modules over $KO$ can then be determined by descent. 
\end{remark} 

\subsection{The Devinatz-Hopkins-Smith nilpotence theorem}

The following result is fundamental for all the applications of nilpotence in
stable homotopy theory: a criterion for the nilpotence of elements
in ring spectra.
Throughout, we work in $\mathcal{C} = \sp$.

We let $MU$ denote the $\e{\infty}$-ring spectrum of \emph{complex bordism.} 
We will also need to use the ring spectra $K(n)$ associated to an implicit
prime $p$ and a height $n$. These are no longer $\e{\infty}$, but they are
$\e{1}$-ring spectra. 
We have $\pi_* K(n) \simeq \mathbb{F}_p[v_n^{\pm 1}]$ with $|v_n| = 2(p^n -1)$. 

\begin{theorem}[Devinatz-Hopkins-Smith \cite{DHS88}] 
\label{DHSthm}
If $R$ is a ring spectrum and $\alpha \in \pi_*(R)$ maps to zero in $MU_*(R)$,
then $\alpha$ is nilpotent. 
\end{theorem}

\begin{corollary}[Hopkins-Smith \cite{HS98}] 
Let $R$ be a ring spectrum and let $\alpha \in \pi_*(R)$. Then the following
are equivalent: 
\begin{enumerate}
\item $\alpha$ is nilpotent. 
\item
The image of $\alpha $ in $K(n)_* (R)$ for each implicit prime $p$ and $0 < n <
\infty$ is nilpotent. Similarly, the image of $\alpha$ in
$(H\mathbb{F}_p)_*( R)$ for each $p$ and $ H\mathbb{Q}_* (R)$ is nilpotent.
\end{enumerate}
\end{corollary}

We remark that the benefit of the above results is that the $K(n)_*$ and
$MU_*$-homology theories are much more computable than stable homotopy
groups itself. That is, while $MU_*$ and $K(n)_*$ are much easier to work with
than $\pi_*$, they are still strong enough to detect nilpotence. 
Our goal in this subsection is to explain how one can interpret the nilpotence
theorem in terms of exponents of nilpotence. 
This subsection is essentially an amplification of remarks of Hopkins \cite{Hopbirthday}.

We remark first that \Cref{DHSthm} would be immediate if we 
knew that $MU$ was \emph{descendable} in spectra. It is also immediate
for ring spectra which are $MU$-nilpotent.
This is not the case, as
there are nontrivial spectra $Y$ with $MU \wedge Y = 0$, for instance the
Brown-Comenentz dual of the sphere (cf. \cite[Lem. B.11]{HoSt}). 
Nonetheless, a partial version of this statement holds in some generality.

As explained in \cite{DHS88}, it suffices to assume that $R$ is
\emph{connective} in \Cref{DHSthm}. 
Let $X$ be a connective spectrum. 
In this case, while $X$ need not be $MU$-nilpotent, the truncations $\tau_{\leq k} X$ are necessarily $MU$-nilpotent
because they can be finitely built from Eilenberg-MacLane spectra. 
The $MU$-exponents of nilpotence of 
$\tau_{\leq k} X$
turn out to give an equivalent formulation of the nilpotence theorem. 
We will in fact formulate a general question for an connective ring spectrum.

For simplicity, we work localized at a prime number $p$. 

\begin{definition} 
Let $R$ be a connective $p$-local ring spectrum with $\pi_0 R =\mathbb{Z}_{(p)}$ and
$\pi_i R$ a finitely generated $\mathbb{Z}_{(p)}$-module for all $i$. 
We define the function $f_R: \mathbb{Z}_{\geq 0} \to \mathbb{Z}_{\geq 0}$ via
the formula
\[ f_R(k) = \exp_R( \tau_{\leq (k-1)} S^0_{(p)}) . \]
\end{definition} 

In general, it is probably impossible to calculate $f_R(k)$ exactly except in
very low degrees, but we ask the following general question.

\begin{question} 
What is the behavior of the function $f_R(k)$ as $k
\to \infty$? 
\end{question}

We note first that the function $f_R$ is \emph{subadditive}, i.e., $f(k_1 + k_2) \leq
f(k_1) +f(k_2)$; this follows from the cofiber sequence
$\tau_{[k_1, k_1+k_2-1]} S^0_{(p)} \to \tau_{ \leq  k_1 + k_2 - 1} S^0_{(p)} \to \tau_{\leq
k_1 -1 } S^0_{(p)}$ and the fact that 
$\tau_{[k_1, k_1+k_2-1]} S^0_{(p)}$ is a module over $\tau_{\leq k_2 -1}
S^0_{(p)}$. 
It follows that
\[ \lim_{k \to \infty} \frac{f_R(k)}{k} \]
exists, and as $f_R(1) = 1$ we find that this limit
is between $0$ and $1$.

\begin{proposition} 
\label{fRk}
The following are equivalent: 
\begin{enumerate}
\item $f_R(k) \leq m$.  
\item Let $X$ be a connective $p$-local spectrum and let $i \leq k-1$. Any
$\alpha \in
\pi_i X$ of  $R$-based Adams filtration of $\alpha$ at least $m$ vanishes.
It suffices to take $X$ with  finitely generated homology. 
\end{enumerate}
\end{proposition} 
\begin{proof} 
This follows from a straightforward diagram chase. 
Suppose $f_R(k) \leq m$ and fix $\alpha \in \pi_i X$ of filtration at
least $m$, for some $i \leq k-1$. 
Let $I_R$ be the fiber of the unit $S^0_{(p)} \to R$. 
By the discussion of the Adams tower in sec. 2, we then have a lifting of
$\alpha: S^i_{(p)} \to X$ through $I_R^{\wedge m} \wedge
X$  and form the diagram
\[ 
\xymatrix{
& I_R^{\wedge m} \wedge X \ar[d] \ar[r]^{p_2} &  I_R^{\wedge m} \wedge X
\wedge \tau_{\leq k-1}
S^0_{(p)}  \ar[d]^0 \\
S^i_{(p)} \ar[r]^{\alpha} \ar[ru] &  X \ar[r]^{p_1} &  X \wedge \tau_{\leq k-1}
S^0_{(p)} 
}
.\]
The hypothesis that $f_R(k) \leq m$ implies that the map $I_R^{\wedge m} \wedge
\tau_{\leq k-1} S^0_{(p)} \to 
\tau_{\leq k-1} S^0_{(p)}$ is nullhomotopic. Hence the smash product with $X$, which is
the right vertical map in the diagram, is also nullhomotopic. 
Since the map $p_1$ is an isomorphism on $\pi_i$, the commutativity of the
diagram now implies that $\alpha$ is nullhomotopic.

Now suppose that the second hypothesis holds. We need to show that 
the map $I_R^{\wedge m} \wedge \tau_{\leq k-1} S^0_{(p)} \to \tau_{\leq k-1}
S^0_{(p)}$
is null. Since $R$ has finitely generated homology,  it
suffices to show that if $F$ is a $p$-local finite connective spectrum with cells in degrees 
up to $k-1$ equipped with a map $F \to I_R^{\wedge m}$, then the composite
\[ \phi: F \to I_R^{\wedge m} \wedge \tau_{\leq k-1} S^0_{(p)} \to \tau_{\leq
k-1} S^0_{(p)},  \]
is nullhomotopic. 

The composite 
$\phi$ is of $R$-based Adams filtration at least $m$. It follows that the
adjoint map $a: S^0_{(p)} \to \tau_{\leq k-1} S^0_{(p)} \wedge \mathbb{D}F$ also has $R$-based
Adams filtration at least $m$. Since the cells of $F$ are in degrees up to
$k-1$, it follows that $\Sigma^{k-1}\mathbb{D} F$ is connective. Thus, we obtain a map 
\[ \Sigma^{k-1} a: S^{k-1}_{(p)} \to  
\Sigma^{k-1} \tau_{\leq k-1} S^0_{(p)} \wedge \mathbb{D}F
\]
which has Adams filtration $\geq m$, and where the target is connective. It
follows that $\Sigma^{k-1}a$ is null by our hypotheses, which implies that $\phi$
is null as desired. 
\end{proof}

\begin{proposition} 
Let $R$ be a connective ring spectrum with $\pi_0 R = \mathbb{Z}_{(p)}$ and suppose that $f_R(k) = o(k)$ as $k \to \infty$.
Then $R$ detects nilpotence, i.e., if $R'$ is a connective $p$-local ring spectrum and
$\alpha \in \pi_k(R')$ maps to zero in $\pi_k(R \wedge R' )$, then $\alpha$ is
nilpotent.
\end{proposition} 
\begin{proof} 

Let $R'$ be a connective $p$-local ring spectrum and let $\alpha \in \pi_n (R')$ map to
zero in $\pi_n (R \wedge R') = \pi_n (R  \wedge  \tau_{\leq n}R')$. 
For each $m$, it follows that $\alpha^m$ is detected in $R$-based Adams filtration at least 
$m$ in $\pi_{nm} (  \tau_{\leq nm} R')$. 
Once $m$ is chosen large enough such that
\[ \exp_{R} (  \tau_{\leq nm} R') \leq f_R(nm+1) < m, \]
it follows that $\alpha^m =0$ by \Cref{fRk}, which proves that $\alpha$ is nilpotent as
desired.
\end{proof}

We then have the following equivalent reformulation of the nilpotence theorem.

\begin{theorem}[Nilpotence theorem, equivalent reformulation] 
We have $f_{MU_{(p)}}(k) = o(k)$ as $k \to \infty$.
In particular, if $X$ is a connective $p$-local spectrum,  then
$\exp_{MU_{(p)}}( \tau_{\leq k} X) = o(k)$ as
$k \to \infty$. 
\label{equivnilpthm}
\end{theorem} 
\begin{proof} 
This will follow from the results of \cite{DHS88} as well as the general
vanishing line arguments of \cite{HS98}. 
Fix a connective $p$-local spectrum $X$ and 
$\epsilon > 0$. 
Taking direct sums, we can assume that $X$ is actually chosen so that it
attains the bound given by $f_{MU_{(p)}}$: that is, we can assume for each $k > 0$,
there exists a nonzero $\alpha_k \in \pi_{k-1} X$ of filtration $f_{MU_{(p)}}(k)$.

We claim that there exist $N, M$ such that 
in the $MU_{(p)}$-based Adams spectral sequence for $X$, we have
$E_N^{s,t} =0 $ if $s > \epsilon(t-s ) + M$. 
It follows from \cite{HPSgeneric} that the class of $X$ with this property is thick. 
By \cite[Prop. 4.5]{DHS88}, it follows that there exists a finite torsion-free
$p$-local spectrum $F$ such that 
$X \wedge F$ has this property (in fact, we can take $N = 2$). It follows that
$X$ has this property in view of the thick subcategory theorem of \cite{HS98}. 

Choose $N, M$ such that we have the vanishing line 
$E_N^{s,t} =0 $ if $s > \epsilon(t-s ) + M$. Then the Adams-Novikov spectral
sequence together with assumption that the $\alpha_k$ exist show that
$f_{MU_{(p)}}(k) \leq \epsilon k + O(1)$. Since $\epsilon$ was arbitrary, it follows
that $f_{MU_{(p)}}(k) = o(k)$ as $k \to \infty$.
\end{proof} 

\begin{remark} 
In \cite{Hopbirthday}, Hopkins explains this result in the following
(equivalent) manner: the Adams-Novikov spectral sequence (i.e., the spectral sequence based on the
cosimplicial object $\cb(MU_{(p)}) \wedge R$), which always converges for $R$
connective, has a ``vanishing curve'' of slope tending to zero  at $E_\infty$. 
That is, there is a function 
$t \mapsto \phi(t)$ with $\phi(t) = o(t)$ such that any 
element in $\pi_t (R)$ of filtration at least $\phi(t)$ vanishes. We can take
$\phi(t) = f_{MU_{(p)}}(t+1)$ in our notation.
Hopkins also raises the more precise question of the behavior of
$f_{MU_{(p)}}(t)$ as
$t \to \infty$, and suggests that $f_{MU_{(p)}}(t) \simeq t^{1/2}$. 
Hopkins actually works in the integral (rather than $p$-local setting). 
\end{remark}

We now explain the situation when $ R = H \mathbb{Z}_{(p)}$. Of course,
$H\mathbb{Z}_{(p)}$ is very
insufficient for detecting nilpotence. So we should expect $f_{H
\mathbb{Z}_{(p)}}(k)
\neq o(k)$. In fact, we can determine its behavior. 

Suppose $X$ is a connective $p$-local spectrum and suppose that $X$ is
$n$-truncated, i.e., $\pi_i X = 0$ if $i \notin [0, n]$. 
Then, by induction on $n$ and the Postnikov filtration, we find easily that
$\exp_{H\mathbb{Z}_{(p)}}(X) \leq n + 1$. 
One can do somewhat better. 
\begin{proposition} 
We have $f_{H\mathbb{Z}_{(p)}}(k) = \frac{k}{2p-2} + O(1)$.
\end{proposition} 
\begin{proof} 
This is based on  the vanishing line in
the $H \mathbb{F}_p$-based Adams spectral sequence. Compare
\cite[Prop. 3.2--3.3]{Mtorsion}. 
The argument there shows that 
$\exp_{H\mathbb{Z}_{(p)}} ( \tau_{\leq k} S^0_{(p)}) \leq \frac{k}{2p-2} + O(1)$.
In fact, it shows the stronger claim that 
$\exp_{H\mathbb{F}_p} ( \tau_{[1, k]} S^0_{(p)}) \leq \frac{k}{2p-2} + O(1)$.

To see the converse, 
we consider the ring spectrum $ku/p$. The element $v_1 \in \pi_{2p-2}( ku/p )$
maps to zero in $\pi_{2p-2}( H \mathbb{Z}_{(p)}  \wedge (ku/p))
\subset \pi_{2p-2}( H \mathbb{F}_p\wedge (ku/p))$. 
However, $v_1$ is not nilpotent. It follows that 
$v_1^n \in \pi_{n(2p-2)} (ku/p)$ has $H\mathbb{Z}_{(p)}$-Adams filtration at least $n$. 
It follows from this 
and \Cref{fRk}
that $f_{H\mathbb{Z}_{(p)}}( n(2p-2) + 1) \geq n+1$. This proves the claim. 
\end{proof}

We can also ask intermediate questions. 
We have a whole family of ring spectra interpolating between $MU$ and $H
\mathbb{Z}_{(p)}$. Since we are working $p$-locally, it is easier to replace
$MU$ with $BP$. We then have the family of ring spectra $BP\left \langle
n\right\rangle$. 
Essentially by construction, $BP\left \langle n\right\rangle$ does not see
$v_{n+1}$-fold periodicities and higher; for instance, $BP\left \langle
n\right\rangle$
is annihilated by the Morava $K$-theories $K(m)$ for $m \geq n+1$. 
This leads to the following question:
\begin{question} 
For $R = BP\left \langle n\right\rangle$, do we have 
$$ \lim_{k \to \infty} \frac{1}{k}\exp_{BP\left \langle
n\right\rangle}( \tau_{\leq k} S^0_{(p)}) = \frac{1}{2(p^{n+1}-2)}?$$
We can also ask for the value of this limit for other intermediate connective
ring spectra, such as $\tau_{\leq n} S^0_{(p)}$ and the $X(n)$-spectra used in
the proof of the nilpotence theorem \cite{DHS88}. 
\end{question}

At $n = 1$, closely related questions have been studied. 
In \cite[Th. 1.1]{mahowaldbo}, Mahowald shows that the $\ko$-based Adams
spectral sequence for the sphere has a vanishing line at $E_\infty$ 
of slope $\frac{1}{5}$.
At odd primes, this has  been considered in the work of
Gonzalez \cite{Gonz}. 
In particular, in \cite{Gonz} it is shown that the $ku$-based Adams spectral
sequence for $S^0$ has a vanishing line at $E_2$ of slope $(p^ 2- p - 1)^{-1}$. 
We can recover the corresponding statement for the exponents as follows.
\begin{example} 
Let $n = 1$ and $p$ be an odd prime. 
Let $\ell$ be the $p$-adic Adams summand of $\widehat{ku}_p$. Then we have a
natural map
\[ \widehat{(S^0)}_p \to \mathrm{fib}( \ell \xrightarrow{\psi^l - 1} \Sigma^{2p-2} \ell ) , \]
which is an equivalence in degrees below degree $2p^2 -2p-2$, where the first
$\beta$-element $\beta_1$ occurs. In fact, the map 
$\widehat{(S^0)}_p \to \mathrm{fib}( \ell \xrightarrow{\psi^l - 1}
\Sigma^{2p-2} \ell )$ detects precisely the image of the $J$-homomorphism, and $\beta_1$ is the
first class in $p$-local stable homotopy which does not belong to the image of
$J$. 

Thus the $p$-completion $\tau_{< (2p^2  -2p-2)} \widehat{S^0}_p$ has exponent $\leq 2$ over $\ell$. 
It follows from this that 
$f_{ku}(2p^2 - 2p - 2) = f_{\ell}( 2p^2 - 2p - 2) \leq 2$. 
In fact, we can apply \Cref{fRk}: if $X$ is a connective $p$-local spectrum
with finitely generated homology and $i 
 < 2p^2 - 2p - 2
$, then any $\alpha \in \pi_i (X)$ of $ku$-filtration 
at least $2$ vanishes. This follows from 
$\exp_{\ell}\left(\tau_{< (2p^2  -2p-2)} \widehat{S^0}_p\right) \leq 2$ and
passage to $p$-completion everywhere. 
As a result, since $f_{ku}(\cdot)$ is subadditive, 
\(  \lim_{k \to \infty} f_{ku}(k) \leq \frac{1}{p^2 - p - 1}.  \)
\end{example} 

\subsection{The Hopkins-Ravenel smash product theorem}

In this subsection, we briefly discuss the role of nilpotence in the smash product theorem of
Hopkins-Ravenel. We refer to \cite{Ravenelnilp} for an account of this result.
See also the course notes \cite{Lchromatic}.

Fix a ring spectrum $E$.
\begin{definition}[Bousfield] 
A spectrum $X$ is \emph{$E$-local} if for all spectra $Y$ with $E \wedge Y $
contractible, we have $[Y, X]_* = 0$.
To any spectrum $X$, we have a universal  approximation $L_E X$
equipped with a map $X \to L_E X $
with the following two properties: 
\begin{enumerate}
\item $X \to L_E X $ becomes an equivalence after smashing with $E$. 
\item $L_E X$ is $E$-local.
\end{enumerate}
This is called the \emph{$E$-localization} of $X$.
\end{definition} 

Since $E$ is a ring spectrum, we see that if $X = E \wedge X'$ for any spectrum
$X'$, then $X$ is $E$-local. 
As $E$-local spectra form a thick subcategory of $\sp$, it follows that 
any $E$-nilpotent spectrum is also $E$-local. 

\begin{definition} 
We say that a localization $L_E$
is \emph{smashing} if for all spectra $X$, the natural map $L_E S^0 \wedge X
\to L_E X $ is an equivalence. This holds if and only if $L_E$ commutes with
arbitrary wedges (i.e., direct sums). 
\end{definition}

Suppose that $E$ is an $\e{1}$-ring spectrum. 
Then we can form the cobar construction $\cb(E) \in \fun(\Delta, \sp)$. 
\begin{proposition} 
Suppose the $\mathrm{Tot}$ tower of the cobar construction $\cb(E)$ is
quickly converging. Then: 
\begin{enumerate}
\item  The inverse limit of $\cb(E)$ is given by $L_E S^0$.
\item  $L_E : \sp \to \sp $ is a smashing localization.
\item The map $L_E S^0$ is $E$-nilpotent. 
\end{enumerate}
\end{proposition} 
\begin{proof} 
Since $\cb(E)$ is quickly converging, it follows that
\[ E \wedge \mathrm{Tot}(\cb(E)) \simeq \mathrm{Tot}(E \wedge \cb(E)) \simeq E  \]
as $E \wedge \cb(E)$ admits a splitting. 
Therefore, the map $S^0 \to \mathrm{Tot}(\cb(E))$ becomes an equivalence after
smashing with $E$. Since $\mathrm{Tot}(\cb(E))$ is clearly $E$-local (as an
inverse limit of $E$-local objects), (1) follows. 
Since the tower is quickly converging, 
it follows that the limit $L_E S^0$  is $E$-nilpotent. 
\end{proof}

Let $E_n$ denote \emph{Morava $E$-theory.} We refer to \cite{rezknotes} for a
treatment of $E_n$ and a proof that it is an $\e{1}$-ring spectrum (in fact,
it is $\e{\infty}$ by the Goerss-Hopkins-Miller theorem). 
We then have the following fundamental result. 

\begin{theorem}[Hopkins-Ravenel] 
The cobar construction $\cb(E_n)$ is quickly converging. 
Therefore, $L_n$ is a smashing localization.
\end{theorem} 

By \Cref{quickconvss}, the Hopkins-Ravenel theorem implies that the $E_n$-based Adams spectral 
sequence for any $L_n S^0$-module  has a horizontal vanishing line at some finite stage. Actually, the
proof of the Hopkins-Ravenel theorem is based on proving this uniform vanishing
line. 
Then one uses the following crucial result. We refer to \cite[Lecture 30]{Lchromatic} for an exposition.

\begin{proposition} 
Suppose $X^\bullet$ is a cosimplicial object. Suppose that there exists $N, h$
such that for every finite spectrum $Y$, the BKSS for $\pi_* (
\mathrm{Tot}(X^\bullet  \wedge Y))$
has the property that $E_{N}^{s,t} = 0$ for $s > N$. 
Then $X^\bullet$ is quickly converging. 
\end{proposition} 

\begin{remark} 
This result is not totally general: it relies on the fact 
that the composite of two \emph{phantom} maps of spectra is zero. 
\end{remark}

\section{Nilpotence and group actions}

In this section, we discuss some of the relationships between nilpotence and
finite group actions, in particular the $\sF$-nilpotence theory of
\cite{MNN15i, MNN15ii}.

\subsection{Group actions}
\label{subsec:nilpaction}
Let $\mathcal{C}$ be a presentably symmetric monoidal stable $\infty$-category. 
Fix a finite group $G$.
We consider $\infty$-category $\fun(BG, \mathcal{C})$ of objects in
$\mathcal{C}$
equipped with a $G$-action. With the tensor product inherited from the
one on $\mathcal{C}$, this is also a presentably 
symmetric monoidal stable $\infty$-category.

\newcommand{\triv}{\mathrm{triv}}
\begin{cons}
We have the 
homotopy fixed point functor 
\[ \fun(BG, \mathcal{C}) \to \mathcal{C}, \quad X \mapsto X^{hG},  \]
which is the (lax symmetric monoidal) right adjoint to the symmetric
monoidal left adjoint functor $\triv: \mathcal{C} \to \fun(BG,
\mathcal{C})$ which equips an object with the trivial action. 

We also have the homotopy orbits functor
\[ \fun(BG, \mathcal{C}) \to \mathcal{C}, \quad X \mapsto X_{hG},  \]
which is the left adjoint to 
$\triv$.
\end{cons}

In general, the homotopy fixed point functor is \emph{not}  a finite homotopy
limit, because $BG$ cannot be modeled by a finite simplicial set. 
This is a basic difference between algebra and homotopy theory: in algebra,
forming the fixed points for a finite group action is a finite inverse limit. 
By contrast, in homotopy theory, the homotopy fixed points $X^{hG}$ can be much
bigger than $X$. (Similarly, $X_{hG}$ is not a finite homotopy colimit.)

\begin{cons}
For any $X \in \fun(BG, \mathcal{C})$, we have a natural transformation 
\[ N_X: X_{hG} \to X^{hG}  \]
called the \textbf{norm map.} We refer to  \cite[Sec. 2.1]{DAGrat} for the construction
in this setting,
though it goes back 
to work of Greenlees-May \cite{GM95}.
The cofiber of the norm map is called the \textbf{Tate construction} $X^{tG}$. 
\end{cons}

%

In this subsection, we will start by describing an instance where  homotopy
fixed points behave more like a finite limit.
\begin{definition} 
\label{dualG}
We let $\mathbb{D}G_+ \in \fun(BG, \sp)$ be the Spanier-Whitehead dual to the
space $G_+$, equipped with the $G$-action by translation. 

The natural symmetric monoidal functor $\sp \to \mathcal{C}$ 
yields a symmetric monoidal functor $\fun(BG, \sp) \to \fun(BG, \mathcal{C})$
and we will write $\mathbb{D}G_+$ also denote the image. 
We have $\mathbb{D}G_+ \in
\clg(\fun(BG, \mathcal{C}))$, i.e., it is naturally a commutative algebra
object.
\end{definition}

\newcommand{\coind}{\mathrm{CoInd}}
\newcommand{\res}{\mathrm{Res}}
\begin{definition} 
\label{Tnil}
We also have a forgetful functor $\res: \fun(BG, \mathcal{C}) \to \mathcal{C}$ which
remembers the underlying object, and this functor has a biadjoint
\[ \coind: \mathcal{C} \to \fun(BG, \mathcal{C}), \quad X \mapsto \prod_{G} X,  \]
called \emph{coinduction.} Given an object $X$ of $\mathcal{C}$, this is the
\emph{indexed} product of $G$ copies of $X$, with $G$-action permuting the
factors. 
\end{definition} 

We will need some basic facts about $\mathbb{D}G_+$: 
\begin{enumerate}
\item For any $X \in \fun(BG, \mathcal{C})$, the object $X \wedge \mathbb{D}G_+ \in
\fun(BG, \mathcal{C})$ has the
property that  
$(X \wedge \mathbb{D}G_+)^{hG} \simeq X$.
In fact, $X \wedge \mathbb{D}G_+ \simeq \coind ( \res X)$,
which easily implies the description
of the homotopy fixed points.
\item  
Since $\coind $ is also 
the \emph{left} adjoint to $\res$, it follows 
that $$\left(X \wedge \mathbb{D}G_+\right)_{hG} \simeq 
\left( \bigvee_G X\right)_{hG} \simeq X,$$
so that the homotopy orbits are also given by $X$. In fact, the Tate
construction is contractible in this case and the norm map is an equivalence. 
\item The $\infty$-category of $\mathbb{D}G_+$-modules in 
$\fun(BG, \mathcal{C})$ is equivalent to $\mathcal{C}$ itself. This follows as
in \cite[Sec. 5.3]{MNN15i}.  
\end{enumerate}

We can now make the main definition of this subsection.

\begin{definition} 
We will say that  an object $X \in \fun(BG, \mathcal{C})$ is \textbf{nilpotent}
if it is $\mathbb{D}G_+$-nilpotent.
\end{definition}

\begin{proposition} 
When $X \in \fun(BG, \mathcal{C})$ is nilpotent, then $X^{tG} =0 $. 
\end{proposition} 
\begin{proof} 
The Tate construction is exact 
and it annihilates any object of the form $Y \wedge \mathbb{D}G_+$. Therefore,
it annihilates $X$ by the evident thick subcategory argument.
\end{proof} 

We have the following basic result. 
All of this is as in Section 2. In the case $\mathcal{C} = \sp$, then we can
identify the homotopy fixed point spectral sequence 
with the $\mathbb{D}G_+$-based Adams spectral sequence in $\fun(BG, \sp)$. 
\begin{proposition} 
When $X \in \fun(BG, \mathcal{C})$ is nilpotent, the following happen: 
\begin{enumerate}
\item  
$X^{hG}$ belongs to the thick
subcategory of $\mathcal{C}$ generated by $X$.
\item
If $\mathcal{C} = \sp$, 
the homotopy fixed point spectral sequence for 
$X^{hG}$, $H^*(G; \pi_* X) \implies \pi_* X^{hG}$, collapses with a horizontal
vanishing line at a finite stage.
\item Any exact functor $\mathcal{C} \to \mathcal{C}$ induces a functor $\fun(BG,
\mathcal{C}) \to \fun(BG, \mathcal{C})$ which preserves nilpotent objects. 
\end{enumerate}
\end{proposition}

Going in reverse, we can also describe the Tate construction in the following manner. See also
\cite{KleinTate}, which uses a version of this as the characterization.
\begin{cons}
\label{tatedesc}
Let $X \in \fun(BG, \mathcal{C})$.
We can write $X \in \fun(BG, \mathcal{C})$ as an $\mathcal{I}$-indexed colimit
$X = \varinjlim X_i$ where $X_i \in \fun(BG, \mathcal{C})$ is nilpotent. 
For example, we can take $X \simeq (X \wedge \mathbb{D}G_+)_{hG}$ where
$\mathbb{D}G_+ \in \fun(BG, \fun(BG, \sp)) \simeq \fun(B(G \times G), \sp)$ via
the right and left multiplication of $G$ on itself. 
Then 
\[ X^{tG} = \mathrm{cofib}( \varinjlim_{\mathcal{I}} X_i^{hG} \to X^{hG} ).  \]
\end{cons}

We now observe that for an algebra object, Tate vanishing is actually
\emph{equivalent} to nilpotence. 

\begin{theorem} 
\label{tatevanishnilp}
Suppose the unit $\mathbf{1} \in \mathcal{C}$ is compact.
Let $R \in \fun(BG, \mathcal{C})$ be an algebra object. 
Then the following are equivalent: 
\begin{enumerate}
\item The Tate construction vanishes, i.e., $R^{tG} =0$. 
\item $R$ is $\mathbb{D}G_+$-nilpotent.
\end{enumerate}
\end{theorem} 
\begin{proof} 
We refer also to \cite[Cor. 10.6]{HHR} for a result in a similar flavor. 
For an abstract result including this (and others), see \cite[Th. 4.19]{MNN15i}. 

\newcommand{\sk}{\mathrm{sk}}
Let $EG_\bullet$ denote the usual simplicial model for the $G$-space $EG$, so
that $EG_n = G^{n+1}$. 
Let $\sk_n EG$ denote the $n$-skeleton of the geometric realization
$|EG_\bullet|$. Then in the category $\fun(BG, \mathcal{C})$, we have an
equivalence
\[ R \simeq \varinjlim_n R \wedge \sk_n EG_+.  \]
Since $R^{tG} = 0$, it follows that we have an equivalence of $R^{hG}$-modules
\[ R^{hG} \simeq\varinjlim_n (R \wedge \sk_n EG_+)^{hG},    \]
and the unit $1 \in \pi_0 R^{hG}$ belongs to the image of the natural map $(R
\wedge \sk_n EG_+)^{hG} \to R^{hG}$ for some $n$ (as $\mathbf{1}$ is compact). 
It follows from this that in $\fun(BG, \mathcal{C})$, $R$ is a retract of $R
\wedge \sk_n EG_+$, which is clearly nilpotent. 
\end{proof}

\begin{example} 
Let $A \to B$ be a faithful $G$-Galois extension, so that $B \in \fun(BG,
\md(A))$. Then $B$ is nilpotent. In fact, we know that $B^{tG} = 0$ \cite[Prop.
6.3.3]{Rog08} so we can apply \Cref{tatevanishnilp}. 
Alternatively, $B \otimes_A B \simeq \prod_G B$ is clearly nilpotent, and then
one can argue that $B$ is nilpotent since $B$ is descendable as an
$A$-algebra by \Cref{galisdesc}. 
\end{example}

\begin{example} 
Suppose multiplication by $|G|$ is an isomorphism on every object in
$\mathcal{C}$ (equivalently, $\pi_* \mathbf{1}$ is a
$\mathbb{Z}[1/|G|]$-algebra). Then Tate constructions in $\mathcal{C}$ vanish. 
\end{example}

We will now study the case $G = C_p$ in some detail. In certain situations, we
will be able to give another description of the $\infty$-category $\fun(BC_p,
\mathcal{C})$ and of what nilpotence entails.

\begin{definition} 
Let $V$ be the one-dimensional complex representation of $C_p$ (choosing a
primitive $p$th root of unity) and let $S^V$ denote its one-point
compactification. 
We have a natural map of $C_p$-spaces $e: S^0 \to S^V$ including the points at $0$ and
$\infty$. We
will use the same notation 
for the associated map in $\fun(BC_p, \sp)$ obtained by taking suspension
spectra.
We write $e^n$ for the induced $n$-fold smash power $S^0 \to S^{nV}$.
\end{definition} 

We now have the following basic result about nilpotence. In various
forms, this argument is crucial to all aspects of this theory. It appears, for
example, in the work of Carlson \cite{Carlson} on Quillen stratification (to be
discussed further in the next section), and in a general form it appears in the
derived induction and restriction theory of \cite{MNN15ii}. 

\begin{proposition} 
Let $X \in \fun(BC_p, \mathcal{C})$. 
Then $X$ is 
nilpotent if and only if $X \wedge e^n: X \to X \wedge S^{nV}$ is nullhomotopic
in $\fun(BC_p, \mathcal{C})$ for some $n \gg 0$.
\end{proposition} 
\begin{proof} 
Note first that $e \wedge \mathbb{D}G_+$ is nullhomotopic. This can be seen
using the self-duality $\mathbb{D}G_+ \simeq G_+$ and the nullhomotopy is at
the level of $G$-spaces itself. 
It follows that $e \wedge \mathbb{D}G_+ \wedge Y$ is nullhomotopic for any $Y
\in \fun(BC_p, \mathcal{C})$, and it follows by a thick subcategory 
argument that if $X \in \fun(BC_p, \mathcal{C})$ is nilpotent, then $e^n \wedge
X$ is null for some $n$.

Conversely, we 
let $S(nV)_+$ denote the unit sphere in the $C_p$-representation $nV =
V^{\oplus n}$ and consider the cofiber sequence
\[ S(nV)_+ \to  S^0 \stackrel{e^n}{ \to } S^{nV}  \]
in $\fun(BC_p, \sp)$, called the \emph{Euler sequence.}
If $X \wedge e^n$ is nullhomotopic, it follows that $X$ is a retract of $X
\wedge S(nV)_+$ in $\fun(BC_p, \mathcal{C})$. Now one sees that this is
nilpotent using the fact that $S(nV)_+$ admits a finite $C_p$-cell decomposition
with free cells. 
\end{proof}

\begin{definition} 
Let $R$ be an $\e{\infty}$-ring.
$R$ is \textbf{complex-oriented} if there exists a map of $\e{1}$-rings $MU \to R$. 
In this case, we have an identification\footnote{Stated another way, if
$BGL_1(R)$ denotes the 
classifying space of rank $1$ $R$-modules, then the composite
\( BC_p \to BGL_1(R)  \) classifying the representation sphere is
nullhomotopic. We refer to \cite{ABGHR1} for a detailed treatment.}
\[ R \wedge S^V  \simeq \Sigma^2 R  \in \fun(BC_p, \md(R)),  \]
so that the map $e$ becomes a map $ R \to \Sigma^2 R $ in $\fun(BC_p, \md(R))$,
which is classified by an element $\beta \in R^{2}(BC_p)$.
\end{definition} 

In this case, we can formulate an equivalence of categories that further
illuminates the condition of nilpotence of a $C_p$-group action. 
We first recall the setup; see also the discussion of
``unipotence'' in \cite{MNN15i}, of which this is another form. 
Let $R$ be an $\e{\infty}$-ring spectrum. Then for any finite group $G$, we have a natural adjunction
of presentably symmetric monoidal stable $\infty$-categories
\[ \md( C^*(BG; R)) \rightleftarrows \fun(BG, \md(R)),  \]
where the right adjoint sends $X \in \fun(BG, \md(R))$ to $X^{hG}$. 
This adjunction is generally far from an equivalence; nonetheless, sometimes it
can come close.

\begin{theorem} 
\label{Cpunipotence}
Suppose $R$ is an  complex-oriented $\e{\infty}$-ring. Then: 
\begin{enumerate}
\item  
The above adjunction induces an equivalence of symmetric monoidal $\infty$-categories $$\fun(BC_p,
\widehat{\mod(R)}_p) \simeq \widehat{\md}( C^*(BC_p; R))_{(p, \beta)}$$ 
between the $\infty$-category of $p$-complete $R$-modules with a $C_p$-action
and the $\infty$-category of $(p, \beta)$-complete 
$C^*(BC_p; R)$-modules.
\item
An object 
of $\fun(BC_p, \widehat{\mod(R)}_p)$
is nilpotent if and only if the $\beta$-action on the corresponding 
$C^*(BC_p; R)$-module 
is nilpotent. 
\end{enumerate}
\end{theorem} 

\begin{proof} 
Consider first the functor
\begin{equation} \label{compfunctorcp} F: {\md( C^*(BC_p; R))}  \to \fun(BC_p,
\mod(R)),\end{equation}
which carries $C^*(BC_p; R)$ to $R$ equipped with 
trivial $C_p$-action. It is clearly fully faithful on the thick subcategory
generated by 
$C^*(BC_p; R)$. Now 
the cofiber of $\beta: R \to \Sigma^2 R$ 
is carried by $F$ to $R \wedge S(V)_+ \in \fun(BC_p, \mod(R))$. 
Since every object in $\fun(BC_p, \md(R))$ is complete with respect to smashing
with $S(V)_+ \in \fun(BC_p, \sp)$, it follows that 
$F$ factors through the $\beta$-completion, so we obtain a functor
\begin{equation}  \label{complcompfunctor2}  \widehat{{\md}}( C^*(BC_p; R))_{\beta}  \to \fun(BC_p,
\mod(R)), \end{equation}
which, again, carries
$C^*(BC_p; R)$ (which is $\beta$-complete) to $R$ equipped with 
trivial $C_p$-action. In particular, it carries the cofiber of $\beta$ to
$R \wedge S(V)_+$. Equivalently, we observe that the right adjoint to $F$,
given by taking $C_p$-homotopy fixed points, lands inside $\beta$-complete
objects. 
We can then 
$p$-complete everywhere to obtain another functor
\begin{equation}  \label{complcompfunctor}  \widehat{{\md}}( C^*(BC_p;
R))_{(p, \beta)}  \to \fun(BC_p,
\widehat{\mod(R)}_p), \end{equation}

Now the cofiber of $\beta$ is a compact generator for 
$\widehat{{\md( C^*(BC_p; R))}}_{\beta}$, and the image $R \wedge S(V)_+ \in
\fun(BC_p, \md(R))$ is compact, since $S(V)_+$ is built from finitely many free
$C_p$-cells. 
Similarly, the iterated cofiber 
$C^*(BC_p; R))/(p, \beta)$ is a compact generator for 
$\widehat{{\md( C^*(BC_p; R))}}_{p, \beta}$, and it corresponds to the compact
object $R \wedge S(V)_+/p \in 
\fun(BC_p, \widehat{\md(R)}_p)$. 

We claim in fact that 
$R \wedge S(V)_+/p$ is a compact \emph{generator} for 
$\fun(BC_p, \widehat{\md(R)}_p)$.
In fact, there is a cofiber sequence
in $\fun(BC_p, \sp)$
\[  C_{p+} \stackrel{x-1}{\to} C_{p+} \to S(V)_+,  \]
where $x$ denotes multiplication by a generator of $C_p$. 
Smashing with $R/p$, we
obtain a cofiber sequence
\[  R/p \wedge C_{p+} \stackrel{x-1}{\to} R/p \wedge C_{p+} \to R/p \wedge S(V)_+,  \]
in $\fun(BC_p, \widehat{\md(R)}_p)$. 
In particular, using the identification 
$\fun(BC_p, \md(R)) \simeq \md( R[C_{p}])$, 
we find that $R/p \wedge S(V)_+$ is the cofiber of multiplication by $x-1$ on
$R[C_p]$.
However,  $x-1$ is nilpotent 
on $R/p \wedge C_{p+} \simeq R[C_p]/p$. 
Therefore, the thick subcategory of 
$\fun(BC_p, \widehat{\md(R)}_p)$ generated by $R/p \wedge S(V)_+$ actually
contains $R/p \wedge
C_{p+}$, which is a compact generator for 
$\fun(BC_p, \widehat{\md(R)}_p)$. It follows that 
$R/p \wedge S(V)_+$ is a compact generator for 
$\fun(BC_p, \widehat{\md(R)}_p)$. 

It follows that 
our comparison functor \eqref{complcompfunctor} carries a compact generator of
the source to a compact generator of the target, and it is fully faithful on
the thick subcategory generated by the compact generator in the source. It
follows now that \eqref{complcompfunctor} is an equivalence, as desired. 

Finally, we need to connect the notions of nilpotence on both sides of
\eqref{complcompfunctor}. 
Note that $X \in \fun(BC_p, \widehat{\md(R)}_p)$ is nilpotent if and only if 
$X$ is a retract of $X \wedge S(nV)_+$ for some $n$.
An object $Y$ of $\widehat{{\md( C^*(BC_p;
R))}}_{p, \beta}$ has nilpotent $\beta$-action  
if and only if $Y$ is a retract of $Y \wedge 
C^*(BC_p;
R)/\beta^n $ for some $n$. Since 
$C^*(BC_p;
\widehat{R}_p)/\beta^n $ and 
$\widehat{R}_p \wedge {S(nV)_+}$
correspond under the equivalence of \eqref{complcompfunctor}, the second
assertion follows. 
\end{proof}

We will now connect this to the Tate construction.
\begin{proposition} 
\label{descTatecons}
Let $R$ be a complex-oriented $\e{\infty}$-ring.
The Tate construction $\fun(BC_p, \md(R)) \to \md(R)$ 
is given by the composite
\[ \fun(BC_p, \md(R)) \xrightarrow{M \mapsto \widehat{M}_p}
\fun(BC_p, \widehat{\md(R)}_p) \simeq \widehat{\md( C^*(BC_p; R))}_{\beta, p}
\xrightarrow{\beta^{-1}} \md(R)
,\]
where the last functor is given by inverting $\beta$. 
In particular, $R^{tC_p} = C^*(BC_p, \widehat{R}_p)[\beta^{-1}]$.
\end{proposition} 
\begin{proof} 
This follows from the description of the Tate construction in
Construction~\ref{tatedesc}.
Note first that the map $M \to \widehat{M}_p$ induces an equivalence on Tate
constructions; this follows from the usual arithmetic square, to which the
other terms do not contribute after applying $()^{tC_p}$. 
Fix $M \in \fun(BC_p, \md(R))$. 
We write $M$ as the colimit 
$M \simeq \varinjlim_n M \wedge S(nV)_+$ in $\fun(BC
_p, \md(R))$. 
It follows that 
\[ M^{tC_p} = \mathrm{cofib}\left( \varinjlim_n (M \wedge S(nV)_+)^{hC_p}  \to
M^{hC_p} \right) 
= \varinjlim_n (M \wedge S^{nV})^{hC_p}
. \]
Now in view of the complex orientation we can identify 
this last colimit with multiplication by $\beta$, i.e.
\[ M^{hC_p} \stackrel{\beta}{\to} \Sigma^2 M^{hC_p} \stackrel{\beta}{\to}
\dots,  \]
which proves the claim.
\end{proof}

As a result, we are able to prove a unipotence result
for associative ring spectra with Tate vanishing. 
The main example is given by the Morava $K$-theories. 
We can view this as a form of the convergence of the Eilenberg-Moore spectral
sequence as in \cite{bauer}.

\begin{definition} 
We will say that a $p$-local ring spectrum $R$ satisfies \textbf{Tate vanishing} if
$R^{tC_p} = 0$. 
\end{definition}

\begin{corollary} 
Suppose $R$ is an $\e{1}$-ring spectrum in $\md(MU)$. Suppose $p$ is nilpotent in $\pi_0 R$. Then the following
are equivalent: 
\begin{enumerate}
\item $R^{tC_p} = 0$.  
\item The trivial $C_p$-action on $R$ is nilpotent.
\item The comparison functor \eqref{compfunctorcp} is an equivalence of
$\infty$-categories $\md( C^*(BC_p; R)) \simeq \fun(BC_p, \md(R))$. 
\item The image of $\beta \in MU^2( BC_p)$ in $R^2(BC_p)$  is nilpotent.
\end{enumerate}
\end{corollary} 
\begin{proof} 

We consider the equivalence
$\widehat{\md}(C^*(BC_p; MU))_{(p, \beta)} \simeq \fun(BC_p,
\widehat{\md(MU)}_p)$ of \Cref{Cpunipotence}. 
We have an $\mathbb{E}_1$-algebra object given on the right-hand-side by $R$ with trivial
$C_p$-action and by $C^*(BC_p; R)$ on the left-hand-side. 
We can thus take module categories with respect to this algebra on both sides. 
Since $p$ is nilpotent in $\pi_0 R$, we thus obtain an equivalence 
\[ \widehat{\md}( C^*(BC_p; R))_\beta \simeq \fun(BC_p, \md(R)).  \]
The desired equivalences now follow in view of this and \Cref{descTatecons}.
For example, if $\beta$ maps to a nilpotent class in $R^2(BC_p)$, then every
module
over 
$\md( C^*(BC_p; R))$ is automatically $\beta$-complete. 
\end{proof} 

\begin{example} 
The Morava $K$-theories $K(n)$ satisfy Tate vanishing. 
In fact, $K(n)$ is complex-oriented and as $K(n)^*(BC_p)$ is a
($\mathbb{Z}/2(p^n-1)$-graded)
finite-dimensional vector space, it follows that any element of the
augmentation ideal, such as $\beta$, is necessarily nilpotent.
Compare \cite{GS96}. 
Compare also \cite[Sec. 5]{HL} for generalizations of these equivalences. 
For example, one can replace $C_p$ by any $p$-group, or in those cases by
appropriate $\pi$-finite spaces. 
\end{example}

\begin{remark} 
It is also known that Tate vanishing holds \emph{telescopically}.
Let $F$ be an associative ring spectrum which is also 
a finite type $n$ complex, and let $x \in \pi_* X$ be a central $v_n$-element
(these exist by \cite{HS98}). Then we let $T(n) = T(n, F, x) = F[x^{-1}]$. 
Then $T(n)$ satisfies Tate vanishing \cite{Kuhn}. 
However, the telescope is very much not complex orientable. 
\end{remark} 
We close with the following natural question suggested by this discussion. 
\begin{question} 
What is the exponent of nilpotence for $K(n) \in \fun(BG, \sp)$ (with trivial
action) as a function of $G$? One can also ask the analogous question for
$T(n)$, although it is likely much harder. 
\end{question}

\subsection{Nilpotence in representation theory}

In the previous subsection, we explored the 
phenomenon of \emph{nilpotence} in $\infty$-categories of the form $\fun(BG,
\mathcal{C})$. In general, the statement that 
\emph{every} object of $\fun(BG, \mathcal{C})$ should be nilpotent is very
strong and is not usually satisfied except in special cases.
However, a weaker form of nilpotence, which involves a family of subgroups of a
given group, is satisfied in a much wider array of cases. In this subsection,
we discuss it primarily in the setting of 
group actions on $k$-modules. 
The material primarily follows \cite{MNN15ii}.

Again, let $\mathcal{C}$ be a presentably symmetric monoidal stable $\infty$-category. 
We will introduce a 
generalization of \Cref{dualG} and \Cref{Tnil}. 
\begin{cons} 
For each $H \leq G$, we have a commutative algebra object $\mathbb{D}(G/H_+) \in \clg(
\fun(BG, \mathcal{C}))$. 
This commutative algebra controls \emph{restriction to $H$} in the following
sense: we have an equivalence of symmetric monoidal $\infty$-categories
\[ \md_{\fun(BG, \mathcal{C})}( \mathbb{D}(G/H_+)) \simeq \fun(BH, \mathcal{C}).  \]
We refer to \cite{BDS} for a discussion of such equivalences (at least at the
level of triangulated categories) as well as \cite[Sec. 5.3]{MNN15i}. 
\end{cons}

\begin{definition} 
Let $\sF$ be a family of subgroups of $G$. 
An object of $\fun(BG, \mathcal{C})$ is \textbf{$\sF$-nilpotent} if it is
nilpotent with respect to $\prod_{H \in \sF} \mathbb{D}(G/H_+)$.
We will say that $\mathcal{C}$ is \textbf{$\sF$-nilpotent} if every object of
$\fun(BG, \mathcal{C})$ is $\sF$-nilpotent.
\end{definition}

When $\sF$ is the trivial family (consisting just of $(1)$), then this of
course recovers \Cref{Tnil}. 
However, when $\sF$ is more general, we expand the class of examples for which 
one has $\sF$-nilpotence. 
One can show that there is always a \emph{unique} minimal family with respect
to which an object is $\sF$-nilpotent. 

\newcommand{\All}{\mathcal{A}ll}
\begin{example} 
Let $\underline{\All}$ be the family  of $p$-groups of $G$ (for any $p$). 
Then $\mathcal{C}$ is automatically $\underline{\All}$-nilpotent. 
In fact, for each $p$ let $G_p \subset G$ be a $p$-Sylow subgroup. 
Then one sees easily that $S^0 \in \fun(BG, \sp)$ is a retract of $\prod_{p
\mid |G|} \mathbb{D}(G/G_{p+})$, which implies that the unit $S^0$ is
$\underline{\All}$-nilpotent, and thus any object is.
\end{example} 

Using this, one can reduce the question of $\sF$-nilpotence to the case where
$G$ is itself a $p$-group.

\begin{proposition} 
\label{Pnil}
Let $X \in \fun(BG, \mathcal{C})$. Then the following are equivalent: 
\begin{enumerate}
\item $X$ is $\mathcal{P}$-nilpotent, where $\mathcal{P}$ is the family of
proper subgroups of $G$. 
\item Let $\widetilde{\rho_G}$ be the (complex) reduced regular representation of $G$,
i.e., the quotient of the regular representation by the trivial representation. 
Let $e_{\widetilde{\rho_G}}: S^0 \to S^{\widetilde{\rho_G}}$ be the associated
Euler class in $\fun(BG, \sp)$. 
Then $X \wedge e^n$ is nullhomotopic for some $n \gg 0$.
\end{enumerate}
\end{proposition}

The first case where this was effectively studied was when $\mathcal{C}$ is the
derived $\infty$-category of a ring, starting with the work of Quillen
\cite{Qui71} and later expanded by Carlson \cite{Carlson}. See also
the work of Balmer \cite{Balmersep}, especially for the context of
descent up to nilpotence.

Fix a finite group $G$ and a field $k$ of characteristic $p> 0$. 
We take $\mathcal{C} = \md(k) = D(k)$. 
We consider the presentably symmetric monoidal $\infty$-category $ \fun(BG, \md(k))$, the $\infty$-category of objects
in $\md(k)$ (or $D(k)$) equipped with a $G$-action, which we can also realize as the
derived $\infty$-category  $D( k[G])$. Using the $k$-linear tensor product,
this is a presentably symmetric monoidal, stable $\infty$-category. 
We will also write $k^{G/H}$ for $\mathbb{D}(G/H)_+$ in here. 

It has been known since the work of Quillen \cite{Qui71} that when working
with phenomena ``up to nilpotence'' in the cohomology of finite groups, the
the \emph{elementary abelian} $p$-subgroups (i.e., subgroups of the form
$C_p^n$ for some $n$) 
play the basic role. This can be formulated in the following result. 

\begin{theorem}[Carlson \cite{Carlson}] 
\label{descentRepthy}
Let $\mathcal{E}_p = \mathcal{E}_p(G)$ be the family of elementary abelian $p$-subgroups of $G$.
Then $\md(k)$ is $\mathcal{E}_p$-nilpotent; that is, 
the algebra $\prod_{H \in \mathcal{E}_p(G)} k^{G/H} \in \clg( \fun(BG, \md(k)))$ is
descendable. 
\end{theorem} 

Carlson's original statement \cite[Th. 2.1]{Carlson} is that there exists 
a filtration
of $G$-representations over $k$
\[ 0  = W_0 \subset W_1 \subset \dots \subset W_n  \]
such that: 
\begin{enumerate}
\item Each quotient $W_i/W_{i-1}$ is induced from some elementary abelian
$p$-subgroup.  
\item The trivial representation $k$ is a retract of $W_n$, i.e., $W_n \simeq
k \oplus W_n'$ for some $G$-representation $W_n'$. 
\end{enumerate}
Clearly, this implies \Cref{descentRepthy}.  

\Cref{descentRepthy} turns out to be closely related to the ``stratification''
results on cohomology of finite groups pioneered by Quillen \cite{Qui71}. 
Recall the statement of Quillen's results. 
Given a finite group $G$ and a subgroup $H \subset G$, we have 
a restriction map $H^*(G; k) \to H^*(H; k)$. 
We let $\mathcal{O}(G)$ denote the \emph{orbit category} of $G$, i.e., the
category of all $G$-sets of the form $G/H, H \subset G$. 
Then we have a functor
\[ \mathcal{O}(G)^{op} \to \mathrm{Ring}, \quad G/H \mapsto H^*(H; k) \simeq 
H^*( G/H \times_G EG; k)
.   \]

\begin{theorem}[Quillen \cite{Qui71}] 
\label{Quithm}
Let $\mathcal{O}_{\mathcal{E}_p}(G) \subset \mathcal{O}(G)$ be the subcategory
of $G$-sets of the form $G/H$ with $H$ an elementary abelian $p$-group. 
The natural map
\[ H^*(G; k) \to \varprojlim_{G/H \in \mathcal{O}_{\mathcal{E}_p}(G)^{op} } H^*(H; k)  \]
is a uniform $\mathscr{F}_p$-isomorphism, that is: 
\begin{enumerate}
\item There exists $N$ such that any element $x \in H^*(G; k)$ in the kernel
satisfies $x^N = 0$. 
\item There exists $M$ such that given any element $y$ of the codomain, $y^{p^M}$ belongs to the image. 
\end{enumerate}
\end{theorem}

\Cref{Quithm} can be recovered from \Cref{descentRepthy} using the Adams
spectral sequence and the machinery of descent up to nilpotence. 
\begin{proof}[Proof sketch of \Cref{Quithm}] 
We claim that this follows from the horizontal vanishing line in the
Adams-based spectral sequence. Namely, we consider $\mathcal{C} = \fun(BG,
\md(k))$ and the algebra object $A = \prod_{H \in \mathcal{E}_p(G) } k^{G/H}$,
which is descendable by \Cref{descentRepthy}. 
Thus, we have a cosimplicial object 
$\cb(A)$
which quickly converges to the unit. 
Taking homotopy in $\mathcal{C}$ (equivalently, forming $\pi_* ()^{hG}$ everywhere)
we thus obtain a spectral sequence converging to the homotopy groups of
$k^{hG} = C^*(BG; k)$. 
The 
spectral sequence collapses with a horizontal vanishing line 
at a finite stage. 
Therefore, anything in positive filtration is nilpotent. 
Moreover, using the Leibniz rule, one sees that any class on the zero-line
$E_2^{0, t}$ survives after applying the Frobenius sufficiently many times.
One can identify the zero-line $E_{2}^{0,t}$
with the inverse limit 
$\varprojlim_{G/H \in \mathcal{O}_{\mathcal{E}_p}(G)^{op} } H^*(H; k)$, which
implies the result. 
\end{proof}

We can ask to make this quantitative. 
For simplicity, we state the problem for the 
family of proper subgroups. 
\begin{question} 
Let $G$ be a finite $p$-group which is not elementary abelian and let $k$ be a
field of characteristic $p$.
What is the exponent of nilpotence of the trivial representation $k$ with respect $A_{\mathcal{P}} = \prod_{H \subsetneq
G} k^{G/H} \in \fun(BG, \md(k))$?  We write this as $n_{\mathcal{P}}(G)$.

There are several equivalent reformulations of this question.
\begin{enumerate}
\item  
Let $Y$ be representation of $G$ and let $\alpha \in H^k(G, Y)$ be a class
which restricts to zero on proper subgroups. 
Then there exists $n$ such that $\alpha^{\otimes n} $ vanishes in $H^{nk}(G;
Y^{\otimes n})$. What is the minimal $n$ (for all $Y$)? 
\item
Let $V$ be the representation $\bigoplus_{H < G} k^{G/H}$ (i.e.,
$A_{\mathcal{P}}$) and consider the
natural short exact sequence of representations 
\[ 0 \to k \to  \bigoplus_{H < G} k^{G/H} \to W \to 0,\]
where $W = V/k$. This defines a natural class in $\mathrm{Ext}^1_{k[G]}(W, k)$, or
equivalently a class $u \in H^1(G, W^{\vee})$. 
There exists $n$ such that $u^{\otimes n} \in H^n(G, W^{\vee, \otimes n})$
vanishes. What is the minimal $n$? 
\item 
What is the minimal dimension of a finite $G$-CW complex $F$ such that $F^G =
\emptyset$ and such that the map $k \wedge F_+ \to k$ in $\fun(BG, \md(k))$
admits a section? Then $n_{\mathcal{P}}(G) = \dim F + 1$. 
Compare \cite[Prop. 2.26]{MNN15ii}. 
\end{enumerate}
\end{question}

Suppose $G$ is not elementary abelian and let 
$I_{\mathrm{ess}} \subset H^*(G; k)$ be the ideal of \emph{essential}
cohomology classes, i.e., those which restrict to zero on all proper subgroups. 
There is a significant literature on the essential cohomology ideal, 
and it was conjectured by Mui (unpublished) and Marx \cite{Marx} that
$I_{\mathrm{ess}}^2 = 0$. 
The conjecture was disproved by Green \cite{Greeness}.

We can interpret the ideal in our framework.
The ideal $I_{\mathrm{ess}}$ consists of the classes of Adams filtration $\geq 1$ in $\pi_*
\mathbf{1}$ in $\fun(BG, \md(k))$ with respect to the algebra
object $A_{\mathcal{P}}$. It follows
from the $A_{\mathcal{P}}$-based Adams spectral sequence that 
$I_{\mathrm{ess}}^{n_{\mathcal{P}}(G)} = 0$.

We can give some upper bounds on $N_{\mathcal{P}}(G)$ as follows.

\begin{example} 
Suppose we have a surjection $\phi: G \twoheadrightarrow G'$ and $G'$ is not
elementary abelian. 
Then $n_{\mathcal{P}}(G) \leq n_{\mathcal{P}}(G')$.
This follows from item (3) above. 
\end{example} 

\begin{example} 
Suppose $G$ has a complex irreducible representation $V$ which is not a
character.
Then $G$ acts on the 
projective space $\mathbb{P}(V)$ without fixed points. 
Then, using the projective bundle formula, one has an equivalence in $\fun(BG,
\md(k))$ given by 
\[  k \wedge \mathbb{P}(V)_+  \simeq \bigoplus_{i = 0}^{\dim_{\mathbb{C}} V }
\Sigma^{2i} k.  \]
Moreover, $\mathbb{P}(V)_+$ has a 
cell decomposition with cells of the form $S^{j} \times G/H_+$ for $H
\subsetneq G$ for $j \leq 2 (\dim_{\mathbb{C}}V - 1)$ by the equivariant triangulation theorem. 
Using a cell decomposition,
one sees that (cf. also \cite[Ex. 5.16]{MNN15ii})
\[ \exp_{A_{\mathcal{P}}}(k) \leq \exp_{A_{\mathcal{P}}}(  k \wedge \mathbb{P}(V)_+)
\leq  2 \dim_{\mathbb{C}}V - 1.
\]
\end{example} 

\begin{example} 
Suppose $\phi: G \twoheadrightarrow C_p$ is a surjection. 
Let $\beta \in H^2(C_p; k)$ be the usual class.
Then the cofiber of $\beta$ in $\fun(BC_p, k)$ has nilpotence
exponent 
$2$ as it can be identified with 
$k \wedge S(V)_+$ as above. 
Thus, the cofiber of $\phi^* \beta$ has $A_{\mathcal{P}}$-exponent $\leq 2$. 
For any family of surjections $\phi_1, \dots, \phi_n$, the cofiber of $\phi^*
\beta_1 \dots \phi^* \beta_n \in H^{2n}(G; k)$ has nilpotence exponent at most $2n$. 
The \emph{cohomological length} $\mathrm{chl}(G)$ of $G$ is the smallest $n$ such that there
exist surjections $\phi_1, \dots, \phi_n$ with 
$\phi^*
\beta_1 \dots \phi^* \beta_n = 0$; Serre's theorem \cite[Prop. 4]{serre}
implies that such an $n$ exists when $G$ is not elementary abelian.
It follows that 
$n_{\mathcal{P}}(G) \leq 2 \mathrm{chl}(G)$. 
We refer to \cite{Yalcin} for upper bounds for $\mathrm{chl}(G)$. 
\end{example} 

In general, as above, there are both geometric and cohomological methods of
obtaining upper bounds on $n_{\mathcal{P}}(G)$. It seems more difficult to
obtain effective \emph{lower} bounds on $n_{\mathcal{P}}(G)$, for example:

\begin{question} 
Is $n_{\mathcal{P}}(G)$ \emph{unbounded} as $G$ varies?
\end{question} 


\subsection{Quillen's theorem over other bases}
We can attempt to replace $k$ with any ring spectrum $R$ here. 
That is, we can consider $\fun(BG, \md(R))$ as a ``brave new'' representation
category, and ask about the analogs of 
\Cref{descentRepthy}. It turns out that a better and more wide-ranging generalization is
to work with ``genuinely'' $G$-equivariant spectra, as in \cite{MNN15i,
MNN15ii}, but we will not treat this generality here. 
\newcommand{\GSpec}{\mathrm{Sp}_G}

We raise the following two general question, to which  the
answer is known at least in a wide variety of special cases. 
\begin{question} 
Let $R$ be an $\mathbb{E}_\infty$-ring spectrum and let $G$ be a finite group. 
What is the minimal family for which  $\md(R)$ is $\sF$-nilpotent? 
If $R$ is an $\e{1}$-ring spectrum, what is the minimal family $\sF$ for which $R \in
\fun(BG, \sp)$ (with trivial action) is nilpotent?\footnote{For the results
below, it is best not only to restrict to $\e{\infty}$-rings as many natural
examples are not (or not known to be) $\e{\infty}$. If $R$ is $\e{\infty}$, the
two statements in the question are equivalent.} 
\end{question}

\begin{remark} 
In the language of \cite{MNN15ii}, the above question is equivalent to asking
for the \emph{derived defect base} of  the Borel-equivariant $G$-spectrum
associated to $R$ (with trivial $G$-action). 
\end{remark}

We have the following two basic analogs over other bases, all proved in
\cite{MNN15ii}. 
The first is a basic ``complex-oriented''
form of \Cref{descentRepthy}, and can be proved in many ways, such as the use of the flag
variety.
One knows also that this cannot be improved; e.g. for $R = MU$ the family of
abelian $p$-subgroups is the best possible. 
The second is closely related to the Hopkins-Kuhn-Ravenel character theory of
\cite{HKR00}.

\begin{theorem} 
\label{nilpabelian}
If $R$ is complex-oriented, then $R \in \fun(BG, \sp)$ is $\sF$-nilpotent for $\sF$ the
family of abelian $p$-subgroups of $G$.
\end{theorem} 
\begin{theorem} 
\label{nilpln}
If $R$ is $L_n$-local, then $R \in \fun(BG, \sp)$ is $\sF$-nilpotent for $\sF$ the
family of abelian rank $\leq n$ $p$-subgroups of $G$.
\end{theorem}

We will sketch the proofs of these two results below. Most of the results in
\cite{MNN15ii} proving $\sF$-nilpotence go through complex orientations, 
and the following is the basic tool.

\begin{proposition}[General reduction step] 
Let $R$ be a complex-oriented associative ring spectrum. Suppose 
$\sF$ is a family of subgroups of $G$. 
Suppose that for any $H \notin \sF$, the natural restriction map $R^*(BH) \to
\prod_{H' \subsetneq H} R^*(BH')$ has nilpotent kernel. Then $R \in \fun(BG,
\sp)$ is
$\sF$-nilpotent. 
\end{proposition} 

\begin{proof} 
We sketch the argument when $\sF$ is the family of \emph{proper} subgroups of
$G$ (to which the general case can be reduced). In this case, by \Cref{Pnil}, it
suffices to consider the Euler class 
$e: S^0 \to S^{\widetilde{\rho_G}}$ of the reduced regular representation; this
is a map in $\fun(BG, \sp)$. 

It suffices to show that 
the map $e^n \wedge R$ in $\fun(BG, \md(R))$ is nullhomotopic for $n \gg 0$. 
But using the complex orientation to identify 
$S^{\widetilde{\rho_G}}$ with $S^{2(|G|-1)}$, $e$
defines a class in $R^{2(|G|-1)}(BG)$ which restricts to zero on all proper
subgroups and which is therefore nilpotent, by assumption. 
\end{proof}

\begin{proof}[Proof of \Cref{nilpabelian} and \Cref{nilpln}] 

It suffices to show that if $G$ is a finite group, and if $x \in R^*(BG)$ is a
class which restricts to zero on all abelian subgroups $A \subset G$, then $x$
is nilpotent. 
To see this, one uses a faithful complex representation of $G$ and the
associated  flag bundle $F \to BG$. One argues that the class $x$ pulls back to
a 
nilpotent class on $F$ because $F$ can be built as a finite homotopy colimit of
spaces of the form $BA$, for $A \subset G$ abelian. However, the map $R^*(BG)
\to R^*(F)$ is injective by the projective bundle formula. 

To prove \Cref{nilpln}, it suffices to treat the case where $R = L_n S^0$.
Using the smash product theorem, one can  now reduce to the case where $R =
E_n$.
One now reduces to the case where $G$ is abelian. It suffices to show that if 
$G$ is of rank $\geq n+1$, then the map $R^*(BG) \to \prod_{G' \subsetneq G}
R^*(BG')$ has nilpotent kernel. In fact, a calculation shows that it is in fact
injectve. 
\end{proof}

Over a field, we saw in the previous subsection how the
descent-up-to-nilpotence picture was enough to imply $\mathscr{F}$-isomorphism
style results after Quillen. It turns out that this works over any ring
spectrum. 
We thus have the following result. 

\begin{theorem}[General $\mathscr{F}$-isomorphism] 
\label{generalfiso}
Suppose $R \in \fun(BG, \sp)$ with trivial action is $\sF$-nilpotent and $R$ is a homotopy commutative and
$\mathbb{E}_1$-ring
spectrum. Then the natural map 
$\phi: R^*(BG) \to \varprojlim_{G/H \in \sOGF^{op}} R^*(BH) $ has the following properties:
\begin{enumerate}
\item $\phi \otimes_{\mathbb{Z}} \mathbb{Z}[1/|G|]$ is an isomorphism.  
\item $\phi_{(p)}$ is a uniform $\mathscr{F}_p$-isomorphism. 
\end{enumerate}
\end{theorem} 

Unlike the proof given in the previous section of \Cref{Quithm}, the proof of 
\Cref{generalfiso} requires some additional work if $R$ is not necessarily
torsion-free. The key piece of input is that the higher terms in the same
Adams-style spectral sequence are all $|G|$-power torsion. This uses some
algebraic facts about Mackey functors. We refer to \cite{MNN15ii} for details. 
The statement after inverting $|G|$ is nontrivial (in general,
calculating $E^*(BG)[1/|G|])$ for $E$ 
a spectrum is a difficult problem) 
and appears in \cite{HKR00} for $R = E_n$ and
related ring spectra, and follows from the spectral sequence as well.

We now note an example of a ring spectrum for which we do not have
$\sF$-nilpotence for any family smaller than the $p$-groups.
\begin{theorem} 
If $R = S^0$, then $\md(R)$ is not $\sF$-nilpotent 
for any family smaller than the family of $p$-subgroups of $G$.
\end{theorem} 
We do not know an ``elementary'' proof of the above result.
For example, for $p$-groups, then we know by the Segal conjecture that the stable cohomotopy of $BG$ is
very close to the $p$-adic completion of the Burnside ring of $G$. As a result,
it is possible to show that the rational conclusion 
of \Cref{generalfiso} fails. 

This raises the following general question. 

\begin{question} 
For which families $\sF$ does there exist an associative ring spectrum $R$ such that
$\sF$
is the minimal family for which  $R \in \fun(BG, \sp)$ (with trivial action) is $\sF$-nilpotent?
\end{question}

\begin{remark}

In all the examples where one has a family smaller than the family of
$p$-groups, the proof relies essentially on complex orientations. In
particular, in all known examples the minimal family (if it is not all
subgroups)  
is contained in the family of \emph{abelian} subgroups. We do not know if there
exists $R$ such that the minimal family is larger than the abelian subgroups
(but not all subgroups). 

\end{remark} 

\bibliographystyle{amsalpha}
\bibliography{Vancouver}

\end{document}